\documentclass[12pt,a4paper,reqno]{amsart}

\usepackage[T1]{fontenc}

\usepackage{amsfonts}          

\usepackage{amsmath}

\usepackage{amssymb}

\usepackage{overpic}
\usepackage{euscript} 
\usepackage{eurosym}
\usepackage{comment}
\usepackage{mathrsfs} 
\usepackage{ifthen}
\usepackage{esint} 
\usepackage{color}

\long\def\symbolfootnote[#1]#2{\begingroup%
\def\thefootnote{\fnsymbol{footnote}}\footnote[#1]{#2}\endgroup}

{\qed\vspace{5pt}}

\usepackage{amsthm}

\newtheoremstyle{lause}
{5pt}
{5pt}
{\slshape}
{\parindent}
{\bfseries}
{.}
{.5em}
{}

\theoremstyle{lause}

\newtheoremstyle{maaritelma}
{5pt}
{5pt}
{\rmfamily}
{\parindent}
{\bfseries}
{.}
{.5em}
{}

\theoremstyle{maaritelma}
\newtheoremstyle{lause}
{5pt}
{5pt}
{\slshape}
{\parindent}
{\bfseries}
{.}
{.5em}
{}

\theoremstyle{lause}
\newtheorem{theorem}{Theorem}[section]
\newtheorem{lemma}[theorem]{Lemma}

\newtheorem{corollary}[theorem]{Corollary}

\newtheoremstyle{maaritelma}
{5pt}
{5pt}
{\rmfamily}
{\parindent}
{\bfseries}
{.}
{.5em}
{}

\theoremstyle{maaritelma}
\newtheorem{definition}[theorem]{Definition}

\newtheorem{example}[theorem]{Example}
\newtheorem{remark}[theorem]{Remark}

\DeclareMathOperator*{\essinf}{ess\,inf}

\numberwithin{equation}{section}

\begin{document}

\thispagestyle{empty}

\begin{center}

{\large{\textbf{Weighted minimum $\alpha$-Green energy problems}}}

\vspace{18pt}

\textbf{Natalia Zorii}

\vspace{18pt}

\emph{In memory of Bent Fuglede (1925\,--\,2023) }\vspace{8pt}

\footnotesize{\address{Institute of Mathematics, Academy of Sciences
of Ukraine, Tereshchenkivska~3, 02000, Kyiv, Ukraine\\
natalia.zorii@gmail.com }}

\end{center}

\vspace{12pt}

{\footnotesize{\textbf{Abstract.} For the $\alpha$-Green kernel $g^\alpha_D$ on a domain $D\subset\mathbb R^n$, $n\geqslant2$, associated with the $\alpha$-Riesz kernel $|x-y|^{\alpha-n}$, where $\alpha\in(0,n)$ and $\alpha\leqslant2$, and a relatively closed set $F\subset D$, we investigate the problem on minimizing the Gauss functional
\[\int g^\alpha_D(x,y)\,d(\mu\otimes\mu)(x,y)-2\int g^\alpha_D(x,y)\,d(\vartheta\otimes\mu)(x,y),\]
$\vartheta$ being a given positive (Radon) measure concentrated on $D\setminus F$, and $\mu$ ranging over all probability measures of finite energy, supported in $D$ by $F$. For suitable $\vartheta$, we find necessary and/or sufficient conditions for the existence of the solution to the problem, give a description of its support, provide various alternative characterizations, and prove convergence theorems when $F$ is approximated by partially ordered families of sets. The analysis performed is substantially based on the perfectness of the $\alpha$-Green kernel, discovered by Fuglede and Zorii (Ann.\ Acad.\ Sci.\ Fenn.\ Math., 2018).
}}
\symbolfootnote[0]{\quad 2010 Mathematics Subject Classification: Primary 31C15.}
\symbolfootnote[0]{\quad Key words: Minimum $\alpha$-Green energy problems with external fields; perfectness of $\alpha$-Green kernels; $\alpha$-Riesz and $\alpha$-Green balayage; $\alpha$-Green equilibrium measure; $\alpha$-thinness of a set at infinity; $\alpha$-harmonic measure on the one-point compactification of $\mathbb R^n$.
}

\vspace{6pt}

\markboth{\emph{Natalia Zorii}} {\emph{Weighted minimum $\alpha$-Green energy problems}}

\section{General conventions and preliminaries}\label{sec1} Fix an (open, connected) domain $D\subset\mathbb R^n$, $n\geqslant2$, a  relatively closed set $F\subset D$, $F\ne D$, and a number $\alpha\in(0,n)$, $\alpha\leqslant2$. This paper deals with minimum $\alpha$-Green energy problems in the presence of external fields over positive (Radon) measures $\mu$ on $D$ with $\mu(D)=1$, supported (in $D$) by $F$, known as the Gauss variational problem (cf.\ e.g.\ Ohtsuka \cite{O}, pertaining to general kernels on a locally compact space).

The $\alpha$-Green kernel $g^\alpha_D(x,y)$, $x\in D$, $y\in\mathbb R^n$, is related to the $\alpha$-Riesz kernel $\kappa_\alpha(x,y):=|x-y|^{\alpha-n}$ ($|x-y|$ being the Euclidean distance between $x,y\in\mathbb R^n$) via the $\alpha$-Riesz balayage of $\varepsilon_x$, the unit Dirac measure at $x\in D$, onto the set $Y:=\mathbb R^n\setminus D$.
The present work is concerned with analytic aspects of the theory of $g^\alpha_D$-pot\-en\-t\-ials, initiated by Frostman \cite{Fr} and Riesz \cite{R} and developed further by Fuglede and Zorii \cite{FZ}, Landkof \cite{L}, and Zorii \cite{Z-AMP} (in case $\alpha=2$, see e.g.\ the monographes by Armitage and Gardiner \cite{AG}, Brelot \cite{Brelo1,Br}, and Doob \cite{Doob}); some details of this theory can also be found below. For the probabilistic counterpart of $g^\alpha_D$-pot\-en\-t\-ial theory we refer to the monographes by Bliedtner and Hansen \cite{BH}, Bogdan, Byczkowski, Kulczycki, Ryznar, Song, and Vondra\v{c}ek \cite{PA}, and Doob \cite{Doob} (see also numerous references therein).

To formulate results related to both $\alpha$-Riesz and $\alpha$-Green kernels, it is convenient to begin with some basic concepts of the theory of potentials with respect to general function kernels $\kappa$ on a locally compact space $X$.

\subsection{Basic concepts of the theory of potentials on a locally compact space}\label{ssec1} For any locally compact (Hausdorff) space $X$, we denote by $\mathfrak M(X)$ the linear space of all (real-valued Radon) measures $\mu$ on $X$, equipped with the so-called {\it vague} ({\it =weak\,$^*$}) topology, i.e.\ the topology of pointwise convergence of the class $C_0(X)$ of all continuous functions $\varphi:X\to(-\infty,\infty)$ with compact support, and by $\mathfrak M^+(X)$ the cone of all positive $\mu\in\mathfrak M(X)$, where $\mu$ is {\it positive} if and only if $\mu(\varphi)\geqslant0$ for every positive $\varphi\in C_0(X)$. A measure $\mu\in\mathfrak M^+(X)$ is said to be {\it bounded} if $\mu(X)<\infty$.

The reader is expected to be familiar with principal concepts of the theory of measures and integration on a locally compact space. For its exposition we refer to Bourbaki \cite{B2} or Edwards \cite{E2}; see also Fuglede \cite{F1} for a brief survey.

For the purposes of this study, it is enough to assume $X$ to have a {\it countable} base of open sets. Then it is
{\it $\sigma$-com\-p\-act} (that is, representable as a countable union of compact sets \cite[Section~I.9, Definition~5]{B1}), see \cite[Section~IX.2, Corollary to Proposition~16]{B3}; and hence
negligibility is the same as local negligibility \cite[Section~IV.5, Corollary~3 to Proposition~5]{B2}. Furthermore, then every measure $\mu\in\mathfrak M(X)$ has a {\it countable} base of vague neighborhoods \cite[Lemma~4.4]{Z23b}, and therefore any vaguely bounded (hence vaguely relatively compact, cf.\ \cite[Section~III.1, Proposition~15]{B2}) set $\mathfrak B\subset\mathfrak M(X)$ has a {\it sequence} $(\mu_j)\subset\mathfrak B$ that converges vaguely to some $\mu_0\in\mathfrak M(X)$.

A {\it kernel\/} on $X$ is thought to be a symmetric, lower semicontinuous (l.s.c.) function $\kappa:X\times X\to[0,\infty]$. Given $\mu,\nu\in\mathfrak M(X)$, the {\it mutual energy} and the {\it potential} with respect to the kernel $\kappa$ are defined by means of the equalities
\begin{align*}
  I_\kappa(\mu,\nu)&:=\int\kappa(x,y)\,d(\mu\otimes\nu)(x,y),\\
  U_\kappa^\mu(x)&:=\int\kappa(x,y)\,d\mu(y),\quad x\in X,
\end{align*}
respectively, provided the value on the right is well defined as a finite number or $\pm\infty$. For $\mu=\nu$, the mutual energy $I_\kappa(\mu,\nu)$ defines the {\it energy} $I_\kappa(\mu,\mu)=:I_\kappa(\mu)$ of $\mu$. For more details about these definitions, see Fuglede \cite[Section~2.1]{F1}.

A kernel $\kappa$ is said to be {\it strictly positive definite} if $I_\kappa(\mu)\geqslant0$ for every (signed) $\mu\in\mathfrak M(X)$, and moreover $I_\kappa(\mu)=0\iff\mu=0$. Then all (signed) $\mu\in\mathfrak M(X)$ of finite energy (i.e., with $I_\kappa(\mu)<+\infty$) form a pre-Hilbert space $\mathcal E_\kappa:=\mathcal E_\kappa(X)$ with the inner product $\langle\mu,\nu\rangle_\kappa:=I_\kappa(\mu,\nu)$ and the energy norm $\|\mu\|_\kappa:=\sqrt{I_\kappa(\mu)}$, see \cite[Section~3.1]{F1}. The topology on $\mathcal E_\kappa$ defined by means of this norm is said to be {\it strong}.

Furthermore, a strictly positive definite kernel $\kappa$ is said to satisfy the {\it consistency principle} (or to be {\it perfect}) if the cone $\mathcal E^+_\kappa:=\mathcal E^+_\kappa(X):=\mathcal E_\kappa\cap\mathfrak M^+(X)$ is {\it complete} in the induced strong topology, and moreover the strong topolology on $\mathcal E^+_\kappa$ is {\it finer} than the vague topology on $\mathcal E^+_\kappa$, see Fuglede \cite[Section~3.3]{F1}. Thus, if a kernel $\kappa$ is perfect, then any strongly Cauchy sequence (net) $(\mu_j)\subset\mathcal E^+_\kappa$ converges {\it both strongly and vaguely} to one and the same measure $\mu_0\in\mathcal E^+_\kappa$, the strong topology on $\mathcal E_\kappa$ as well as the vague topology on $\mathfrak M(X)$ being Hausdorff.

For any $A\subset X$, we denote by $\mathfrak M^+(A;X)$ the set of all $\mu\in\mathfrak M^+(X)$ {\it concentrated on} $A$, which means that $A^c:=X\setminus A$ is $\mu$-negligible, or equivalently that $A$ is $\mu$-mea\-s\-ur\-ab\-le and $\mu=\mu|_A$, $\mu|_A$ being the trace of $\mu$ to $A$, cf.\ \cite[Section~V.5.7]{B2}. For any $\mu\in\mathfrak M^+(A;X)$, then necessarily $\mu(X)=\mu(A)$. Also note that if $A$ is closed, then $\mu\in\mathfrak M^+(A;X)$ if and only if $S(\mu;X)\subset A$, where $S(\mu;X)$ is the support of $\mu$ in $X$.

Let $\mathfrak M^+(A,q;X)$, $q>0$, consist of all $\mu\in\mathfrak M^+(A;X)$ with $\mu(X)=q$, and let
\[
\mathcal E^+_\kappa(A):=\mathcal E_\kappa\cap\mathfrak M^+(A;X),\quad
\mathcal E^+_\kappa(A,q):=\mathcal E_\kappa\cap\mathfrak M^+(A,q;X).
\]

All the sets, to appear in this paper, are {\it Borel}, while all the kernels are {\it perfect}. By virtue of \cite[Theorem~3.5]{Z-EM} (cf.\ \cite[Theorem~4.5]{F1}), all those sets are therefore {\it capacitable}, and hence their {\it inner} and {\it outer capacities} coincide:
\[\underline{c}_\kappa(A)=\overline{c}_\kappa(A)=:c_\kappa(A).\]
(For the theory of inner and outer capacities of arbitrary sets in a locally compact space endowed with a perfect kernel, we refer to Fuglede \cite{F1}, cf.\ also Zorii \cite{Z-arx-22}.)

Thus,\footnote{As usual, the infimum over the empty set is interpreted as $+\infty$. We also agree that $1/0=+\infty$ and $1/(+\infty)=0$.}
\begin{equation}\label{cap-def}
c_\kappa(A):=\Bigl[\inf_{\mu\in\mathcal E^+_\kappa(A,1)}\,\|\mu\|_\kappa^2\Bigr]^{-1},
\end{equation}
and therefore (cf.\ \cite[Lemma~2.3.1]{F1})
\begin{equation}\label{ifff}
c_\kappa(A)=0\iff\mathcal E^+_\kappa(A)=\{0\}.
\end{equation}
Alternatively (see \cite[Theorem~4.1]{F1} and \cite[Theorem~6.1]{Z-arx-22}),
\begin{equation}\label{cap-def1}
c_\kappa(A)=\inf_{\mu\in\Gamma^+_\kappa(A)}\,\|\mu\|_\kappa^2,
\end{equation}
where $\Gamma^+_\kappa(A)$ consists of all $\mu\in\mathcal E^+_\kappa$ such that $U_\kappa^\mu\geqslant1$ q.e.\ ({\it quasi-everywhere}) on $A$ (that is, on all of $A$ except for a subset of capacity zero).\footnote{It is used here that
for every $\mu\in\mathfrak M^+(X)$, $U_\kappa^\mu$ is l.s.c.\ on $X$ (see \cite[Lemma~2.2.1(c)]{F1}), and therefore $U_\kappa^\mu$ is $\nu$-measurable for every $\nu\in\mathfrak M^+(X)$, cf.\ \cite[Section~IV.5, Corollary~3 to Theorem~2]{B2}.}

If $c_\kappa(A)<\infty$, then the infimum in  (\ref{cap-def1}) is attained at the unique $\gamma_{A,\kappa}\in\Gamma^+_\kappa(A)$, called the {\it $\kappa$-cap\-ac\-i\-tary distribution} for $A$ (Fuglede \cite[Theorem~4.1]{F1}).\footnote{Unless $\mathcal E^+_\kappa(A)$ is strongly closed (which occurs, in particular, if $A$ is closed or even quasiclosed, see \cite[Definition~2.1]{F71} and \cite[Theorem~3.1]{Z-EM} for details), $\gamma_{A,\kappa}$ is {\it not} necessarily concentrated on $A$, and hence problem (\ref{cap-def}) is in general unsolvable.} If moreover Frostman's maximum principle holds, then $\gamma_{A,\kappa}$ is the only measure in $\Gamma^+_\kappa(A)$ with
\begin{equation}\label{eqpot}
U^{\gamma_{A,\kappa}}_\kappa=1\quad\text{q.e.\ on $A$};
\end{equation}
in this case, $\gamma_{A,\kappa}$ is also referred to as the {\it $\kappa$-equilibrium measure} for $A$. ({\it Frostman's maximum principle} means that for any $\mu\in\mathfrak M^+(X)$ with $U^\mu_\kappa\leqslant c_\mu$ $\mu$-a.e., where $c_\mu\in(0,\infty)$, the same inequality holds on all of $X$. See e.g.\ Ohtsuka \cite[p.~143]{O}.)

All the kernels, to appear in this paper, satisfy the Frostman and {\it domination maximum principles}, where the latter means that for any $\mu,\nu\in\mathfrak M^+(X)$ such that both $U_\kappa^\mu<\infty$ and $U_\kappa^\mu\leqslant U_\kappa^\nu$ hold true $\mu$-a.e., we have
$U_\kappa^\mu\leqslant U_\kappa^\nu$ on all of $X$.

A measure $\mu\in\mathfrak M^+(X)$ is said to be {\it $c_\kappa$-absolutely continuous} if $\mu(K)=0$ for every compact set $K\subset X$ with $c_\kappa(K)=0$. Each $\mu\in\mathfrak M^+(X)$ of finite $\kappa$-energy is certainly $c_\kappa$-absolutely continuous, cf.\ (\ref{ifff}), but not the other way around. With regard to the latter, see the example in Landkof \cite[pp.~134--135]{L}, pertaining to the Newtonian kernel $\kappa_2(x,y):=|x-y|^{2-n}$ on $\mathbb R^n$, $n\geqslant3$.

In the rest of this subsection, a set $A\subset X$ is {\it closed}. Then the cone $\mathfrak M^+(A;X)$ is vaguely closed \cite[Section~III.2, Proposition~6]{B2}, which entails, by use of the perfectness of the kernel $\kappa$, that the subcone
$\mathcal E^+_\kappa(A)$ of the strongly complete cone $\mathcal E^+_\kappa$ is {\it strongly closed}, whence {\it strongly complete}. By utilizing Edwards \cite[Theorem~1.12.3]{E2}, this implies, in turn, that for every $\lambda\in\mathcal E^+_\kappa$, there exists the unique {\it orthogonal projection} $\lambda_\kappa^A$ of $\lambda$ onto $\mathcal E^+_\kappa(A)$, minimizing $\|\lambda-\mu\|_\kappa$ over all $\mu\in\mathcal E^+_\kappa(A)$. This $\lambda_\kappa^A$ actually serves as the {\it balayage} of $\lambda\in\mathcal E^+_\kappa$ onto the set $A$, for, according to \cite[Theorem~4.3, Corollary~4.5]{Z22}, it is uniquely characterized within $\mathcal E^+_\kappa(A)$ by means of the equality
\begin{equation*}
U_\kappa^{\lambda_\kappa^A}=U_\kappa^\lambda\quad\text{q.e.\ on $A$}.
\end{equation*}

\begin{theorem}\label{th-comp} Given a closed set $A$ and a perfect kernel $\kappa$, assume $c_\kappa(A)<\infty$.
Then for any $q\in(0,\infty)$, $\mathcal E^+_\kappa(A,q)$ is complete in the induced strong topology.
\end{theorem}

\begin{proof}
  See Zorii \cite[Theorem~7.1]{Z-EM}. (This certainly no longer holds if $c_\kappa(A)=\infty$.)
\end{proof}

Let $\infty_X$ be the Alexandroff point of a locally compact space $X$ \cite[Section~I.9.8]{B1}.

\subsection{$\alpha$-Riesz balayage}\label{ssec2} In this and the following two subsections, $X:=\mathbb R^n$ and $\kappa:=\kappa_\alpha$, where $n$, $\alpha$, and $\kappa_\alpha$ are as indicated at the top of Section~\ref{sec1}. When speaking of $\nu\in\mathfrak M^+(\mathbb R^n)$, we always understand that its $\kappa_\alpha$-potential $U_{\kappa_\alpha}^\nu$ is not identically infinite on $\mathbb R^n$, which according to Landkof \cite[Section~I.3.7]{L} occurs if and only if
\begin{equation}\label{intf}
\int_{|y|>1}\frac{d\nu(y)}{|y|^{n-\alpha}}<\infty.
\end{equation}
Then, actually, $U_{\kappa_\alpha}^\nu$ is finite q.e.\ on $\mathbb R^n$ \cite[Section~III.1.1]{L}. Note that (\ref{intf}) necessarily holds if $\nu$ is either bounded, or of finite $\kappa_\alpha$-energy $I_{\kappa_\alpha}(\nu)$; with regard to the latter, see \cite[Corollary to Lemma~3.2.3]{F1} applied to $\kappa_\alpha$, the kernel $\kappa_\alpha$ being strictly positive definite according to Riesz \cite[Section~I.4, Eq.~(13)]{R}.

Moreover, $\kappa_\alpha$ is perfect and meets the Frostman and domination maximum principles. See Deny \cite[p.~121, Corollary~(a)]{D1}, Doob \cite[Theorem~1.V.10(b)]{Doob}, Landkof \cite[Lemma~1.3, Theorems~1.10, 1.29]{L}; cf.\ Section~\ref{ssec1} above for definitions.

Referring to Bliedtner and Hansen \cite{BH} (resp.\ Zorii \cite{Z-bal,Z-bal2}) for a general theory of outer (resp.\ inner) $\kappa_\alpha$-bal\-ay\-age of any $\nu\in\mathfrak M^+(\mathbb R^n)$ to any $Q\subset\mathbb R^n$, in this study we limit ourselves to {\it closed} $Q$; this limitation will generally not be repeated henceforth.

Let $Q^r$ denote the set of all {\it $\alpha$-regular} points of $Q$; by the Wiener type criterion \cite[Theorem~5.2]{L},
\begin{equation*}x\in Q^r\iff\sum_{j\in\mathbb N}\,\frac{c_{\kappa_\alpha}(Q_j)}{q^{j(n-\alpha)}}=\infty,\end{equation*}
where $q\in(0,1)$ and $Q_j:=Q\cap\bigl\{y\in\mathbb R^n:\ q^{j+1}<|x-y|\leqslant q^j\bigr\}$. The set $Q^r$ is Borel measurable \cite[Theorem~5.2]{Z-bal2}, while $Q\setminus Q^r$, the set of all {\it $\alpha$-ir\-reg\-ul\-ar} points of $Q$, is of $\kappa_\alpha$-capacity zero (the Kel\-logg--Ev\-ans type theorem, see \cite[Theorem~6.6]{Z-bal}).

\begin{definition}\label{def-bal}Given a (closed) set $Q\subset\mathbb R^n$, fix $\xi\in\mathfrak M^+(\mathbb R^n)$, $\xi\ne0$, such that $\xi|_{Q^r}$ is $c_{\kappa_\alpha}$-absolutely continuous. The {\it $\kappa_\alpha$-balayage} of $\xi$ onto $Q$ is the unique $c_{\kappa_\alpha}$-absolutely continuous measure $\xi_{\kappa_\alpha}^Q\in\mathfrak M^+(Q;\mathbb R^n)$ having the property
\begin{equation}\label{bal1}
U_{\kappa_\alpha}^{\xi_{\kappa_\alpha}^Q}=U_{\kappa_\alpha}^\xi\quad\text{on\ $Q^r$}\quad\text{(hence, q.e.\ on $Q$)}.
\end{equation}
\end{definition}

Regarding the existence and uniqueness of this $\xi_{\kappa_\alpha}^Q$, see Zorii \cite[Corollary~5.2]{Z-bal2}. (Compare with Fuglede and Zorii \cite[Corollary~3.19]{FZ}, pertaining to $\xi\in\mathfrak M^+(Q^c;\mathbb R^n)$. Being based on Theorem~3.17 of the same paper, it however had a gap in its proof; see Remark~\ref{compare} below for details.)

In the rest of this subsection, $\xi$ and $\xi_{\kappa_\alpha}^Q$ are as indicated in Definition~\ref{def-bal}. Since $\xi_{\kappa_\alpha}^Q$ is $c_{\kappa_\alpha}$-abs\-olutely continuous, (\ref{bal1}) holds true $\xi_{\kappa_\alpha}^Q$-a.e. By the domination principle,
\begin{equation*}
U_{\kappa_\alpha}^{\xi_{\kappa_\alpha}^Q}\leqslant U_{\kappa_\alpha}^\xi\quad\text{on all of $\mathbb R^n$},
\end{equation*}
whence, by the principle of positivity of mass for $\kappa_\alpha$-potentials \cite[Theorem~3.11]{FZ},
\begin{equation}\label{bal4}
\xi_{\kappa_\alpha}^Q(\mathbb R^n)\leqslant\xi(\mathbb R^n).
\end{equation}

In the theory of $\kappa_\alpha$-balayage, the following integral representation is particularly useful (see \cite[Theorem~5.1]{Z-bal2}):
\begin{equation}\label{int-repr}
\xi_{\kappa_\alpha}^Q=\int(\varepsilon_x)_{\kappa_\alpha}^Q\,d\xi(x),
\end{equation}
where
\[(\varepsilon_x)_{\kappa_\alpha}^Q:=\varepsilon_x\quad\text{for all $x\in Q^r$}.\]

\begin{remark}\label{compare}
Compare with \cite[Theorem~3.17]{FZ}, dealing with $\xi\in\mathfrak M^+(Q^c;\mathbb R^n)$. However, this theorem from \cite{FZ} was based on Lemma~3.16 of the same paper, whose proof was incomplete. Indeed, the functions $y\mapsto\int f_j\,d\varepsilon_y^A$, appearing in \cite[Proof of Lemma~3.16(a)]{FZ}, might be of noncompact support, and hence the suggested use of \cite[Section~IV.3, Proposition~4]{B2} (see the last line on p.~133 in \cite{FZ}) was unjustified.
\end{remark}

Thus, in view of (\ref{int-repr}), for every $\varphi\in C_0(\mathbb R^n)$,
\begin{equation}\label{I}
\int\varphi(z)\,d\xi_{\kappa_\alpha}^Q(z)=\int\biggl(\int\varphi(z)\,d(\varepsilon_x)^Q_{\kappa_\alpha}(z)\biggr)\,d\xi(x).
\end{equation}
Moreover, by virtue of \cite[Section~V.3, Proposition~2]{B2}, equality (\ref{I}) remains valid when $\varphi$ is replaced by any positive l.s.c.\ function on $\mathbb R^n$. For a given $y\in\mathbb R^n$, we apply this to the (positive l.s.c.) function $\kappa_\alpha(y,z)$, $z\in\mathbb R^n$, and thus obtain
\begin{equation}\label{repr-th1}
U^{\xi_{\kappa_\alpha}^Q}_{\kappa_\alpha}(y)=\int\biggl(\int\kappa_\alpha(y,z)\,d(\varepsilon_x)^Q_{\kappa_\alpha}(z)\biggr)\,d\xi(x)=\int U^{(\varepsilon_x)^Q_{\kappa_\alpha}}_{\kappa_\alpha}(y)\,d\xi(x).
\end{equation}
Similarly, applying (\ref{I}) to the (positive l.s.c.) constant function $1$ on $\mathbb R^n$ gives
\begin{equation}\label{TM}
 \xi_{\kappa_\alpha}^Q(\mathbb R^n)=\int(\varepsilon_x)^Q_{\kappa_\alpha}(\mathbb R^n)\,d\xi(x).
\end{equation}

\subsection{$\alpha$-thinness of a set at infinity}\label{sec-thin} By Kurokawa and Mizuta \cite{KM}, a (closed) set $Q\subset\mathbb R^n$ is said to be {\it $\alpha$-thin at infinity} if
\[\sum_{j\in\mathbb N}\,\frac{c_{\kappa_\alpha}(Q_j)}{q^{j(n-\alpha)}}<\infty,\]
where $q\in(1,\infty)$ and $Q_j:=Q\cap\bigl\{y\in\mathbb R^n:\ q^j<|y|\leqslant q^{j+1}\bigr\}$.

\begin{theorem}\label{th-thin} The following {\rm(i)--(vi)} are equivalent.
\begin{itemize}
\item[{\rm(i)}] $Q$ is $\alpha$-thin at infinity.
\item[{\rm(ii)}] For some {\rm(equivalently, every)} $x\in\mathbb R^n$,
\[x\not\in(Q^*_x)^r,\]
where $Q^*_x$ is the inversion of\/ $Q\cup\{\infty_{\mathbb R^n}\}$ with respect to $\{y:\ |y-x|=1\}$.
\item[{\rm(iii)}]There exists $\nu\in\mathfrak M^+(\mathbb R^n)$ such that
\[\essinf_{x\in Q}\,U_{\kappa_\alpha}^\nu(x)>0,\]
the infimum being taken over all of $Q$ except for a subset of $\kappa_\alpha$-capacity zero.
\item[{\rm(iv)}] There exists the unique measure $\gamma_{Q,\kappa_\alpha}\in\mathfrak M^+(Q;\mathbb R^n)$ having the property
\begin{equation*}
U_{\kappa_\alpha}^{\gamma_{Q,\kappa_\alpha}}=1\quad\text{q.e.\ on $Q$};
\end{equation*}
such $\gamma_{Q,\kappa_\alpha}$ is said to be the {\rm(}generalized\/{\rm)} $\kappa_\alpha$-equilibrium measure on $Q$.\footnote{Both $I_{\kappa_\alpha}(\gamma_{Q,\kappa_\alpha})$ and $\gamma_{Q,\kappa_\alpha}(Q)$ may be $+\infty$; these are finite if and only if so is $c_{\kappa_\alpha}(Q)$. Also note that $\gamma_{Q,\kappa_\alpha}$ is $c_{\kappa_\alpha}$-absolutely continuous, for $U_{\kappa_\alpha}^{\gamma_{Q,\kappa_\alpha}}\leqslant1$ on $\mathbb R^n$ by Frostman's maximum principle.}
\item[{\rm(v)}] There exists $x\in Q^c$ such that $(\varepsilon_x)^Q_{\kappa_\alpha}(\mathbb R^n)<1$.\footnote{For such $x$, $(\varepsilon_x)^Q_{\kappa_\alpha}(\mathbb R^n)=U_{\kappa_\alpha}^{\gamma_{Q,\kappa_\alpha}}(x)$, $\gamma_{Q,\kappa_\alpha}$ being introduced by (iv). See \cite[Theorem~2.2]{Z-bal2}.}
\item[{\rm(vi)}] There exists $\xi\in\mathfrak M^+(Q^c;\mathbb R^n)$ for which strict inequality in {\rm(\ref{bal4})} holds:
\begin{equation}\label{bal5}\xi_{\kappa_\alpha}^Q(\mathbb R^n)<\xi(\mathbb R^n).\end{equation}
\end{itemize}
\end{theorem}

\begin{proof}See \cite[Theorem~5.1]{L}, \cite[Theorems~5.5, 8.6, 8.7]{Z-bal}, and \cite[Theorems~2.1, 2.2, Corollary~5.3]{Z-bal2}. We emphasize that in \cite[Proofs of Theorems~8.6, 8.7]{Z-bal}, one has to use \cite[Theorem~5.1]{Z-bal2} in place of \cite[Theorem~8.2]{Z-bal} (cf.\ Remark~\ref{compare} above).\end{proof}

\begin{remark}\label{refined}
Unless $\alpha<2$, assume that $Q^c$ is connected. Then (v) and (vi) in Theorem~\ref{th-thin} can be refined as follows (see \cite[Theorem~8.7]{Z-bal}, cf.\ (\ref{TM})):
\begin{itemize}
\item[{\rm(v$'$)}] {\sl $Q$ is $\alpha$-thin at infinity if and only if for some {\rm(}equivalently, for all\/{\rm)} $x\in Q^c$, we have $(\varepsilon_x)^Q_{\kappa_\alpha}(\mathbb R^n)<1$.}
\item[{\rm(vi$'$)}] {\sl $Q$ is $\alpha$-thin at infinity if and only if for some {\rm(}equivalently, for all\/{\rm)} nonzero $\xi\in\mathfrak M^+(Q^c;\mathbb R^n)$, {\rm(\ref{bal5})} holds true.}
\end{itemize}
\end{remark}

\subsection{$\alpha$-harmonic measure on $\overline{\mathbb R^n}:=\mathbb R^n\cup\{\infty_{\mathbb R^n}\}$}\label{ssec3} We call a set $e\subset\overline{\mathbb R^n}$ {\it Borel\/} if so is $e\cap\mathbb R^n$.
For any open $\Delta\subset\mathbb R^n$ and any Borel $e\subset\overline{\mathbb R^n}$, we define the (fractional) {\it $\alpha$-harmonic measure} $\omega_\alpha(x,e;\Delta)$, $x\in\Delta$, by means of the formula (see \cite[Section~3.2]{Z-AMP})
\begin{equation}\label{def-h}
\omega_\alpha(x,e;\Delta)=\left\{
\begin{array}{cl}(\varepsilon_x)^{\Delta^c}_{\kappa_\alpha}(e)&\text{\ if $e\subset\mathbb R^n$},\\
(\varepsilon_x)^{\Delta^c}_{\kappa_\alpha}(e\cap\mathbb R^n)+\omega_\alpha(x,\{\infty_{\mathbb R^n}\};\Delta)&\text{\ otherwise},\\ \end{array} \right.\end{equation}
where
\begin{equation}\label{Def}
\omega_\alpha(x,\{\infty_{\mathbb R^n}\};\Delta):=1-(\varepsilon_x)^{\Delta^c}_{\kappa_\alpha}(\mathbb R^n).
\end{equation}
Thus,
\[\omega_\alpha(x,\overline{\mathbb R^n};\Delta)=1\quad\text{for all $x\in\Delta$}.\]
Since for $\alpha=2$, $S\bigl((\varepsilon_x)^{\Delta^c}_{\kappa_2};\mathbb R^n\bigr)\subset\partial_{\mathbb R^n}\Delta$ for every $x\in\Delta$,
cf.\ \cite[Theorem~8.5]{Z-bal}, the concept of $\alpha$-harmonic measure in $\overline{\mathbb R^n}$, introduced by (\ref{def-h}) and (\ref{Def}), generalizes that of $2$-har\-mo\-nic measure, defined in \cite[Section~IV.3.12]{L} for Borel subsets of $\partial_{\mathbb R^n}\Delta$.

It is clear from (\ref{bal4}) that $(\varepsilon_x)^{\Delta^c}_{\kappa_\alpha}(\mathbb R^n)\leqslant1$ for every $x\in\Delta$, whence
\begin{equation*}\omega_\alpha(x,\{\infty_{\mathbb R^n}\};\Delta)\geqslant0\quad\text{for all $x\in\Delta$.}\end{equation*}
On account of Remark~\ref{refined}(v$'$), we arrive at the following observation.

\begin{corollary}\label{cor-infty} Unless $\alpha<2$, assume that $\Delta$ is connected. Then $\Delta^c$ is not $\alpha$-thin at infinity if and only if for some {\rm(}equivalently, for all\/{\rm)} $x\in\Delta$,
\[\omega_\alpha(x,\{\infty_{\mathbb R^n}\};\Delta)=0.\]
\end{corollary}

\subsection{$\alpha$-Green kernel}\label{ssec4} In the rest of this paper, we fix a domain $D\subset\mathbb R^n$.
A measure $\mu\in\mathfrak M^+(D)$, $D$ being treated as a locally compact space, is said to be {\it extendible} by zero outside $D$ to all of $\mathbb R^n$ if there is the (unique) $\widehat{\mu}\in\mathfrak M^+(\mathbb R^n)$ such that
\begin{equation*}\int\varphi|_D\,d\mu=\widehat{\mu}(\varphi)\quad\text{for all $\varphi\in C_0(\mathbb R^n)$};\end{equation*}
let $\breve{\mathfrak M}^+(D)$ denote the class of all those $\mu$. If no confusion can arise, this extension of $\mu\in\breve{\mathfrak M}^+(D)$ will be denoted by the same symbol $\mu$, i.e.\ $\widehat{\mu}:=\mu$.

Note that $\breve{\mathfrak M}^+(D)$ can equivalently be introduced as the class of all traces $\nu|_D$, $\nu$ ranging over all of $\mathfrak M^+(\mathbb R^n)$. We also remark that, in general, $\breve{\mathfrak M}^+(D)\ne\mathfrak M^+(D)$ (unless, of course, $Y:=\mathbb R^n\setminus D$ is compact). Another useful observation is that a measure $\mu\in\mathfrak M^+(D)$ is obviously extendible if it is bounded. For any $A\subset D$, denote
\[\breve{\mathfrak M}^+(A;D):=\mathfrak M^+(A;D)\cap\breve{\mathfrak M}^+(D).\]

The {\it $\alpha$-Green kernel} $g_\alpha:=g_D^\alpha$ on $D$ is defined by
\begin{equation}\label{g}
g_\alpha(x,y):=\kappa_\alpha(x,y)-U_{\kappa_\alpha}^{(\varepsilon_x)_{\kappa_\alpha}^Y}(y)=
U_{\kappa_\alpha}^{\varepsilon_x}(y)-
U_{\kappa_\alpha}^{(\varepsilon_x)_{\kappa_\alpha}^Y}(y),\quad x\in D, \ y\in\mathbb R^n,
\end{equation}
where $(\varepsilon_x)_{\kappa_\alpha}^Y$ is the $\kappa_\alpha$-balayage of $\varepsilon_x$, $x\in D$, onto $Y$, cf.\ Definition~\ref{def-bal}. (Here we have used the fact that the measure $\varepsilon_x\in\mathfrak M^+(D)$, $x\in D$, is extendible, and hence it can be treated as an element of $\mathfrak M^+(\mathbb R^n)$.) For more details, see e.g.\ \cite{Fr,FZ,L}.

\begin{lemma}\label{gpot}
For any $\mu\in\breve{\mathfrak M}^+(D)$, the $g_\alpha$-potential $U^\mu_{g_\alpha}$ is well defined and finite q.e.\ on $\mathbb R^n$ and given by
\begin{equation}\label{g1}
U^\mu_{g_\alpha}=U_{\kappa_\alpha}^\mu-U_{\kappa_\alpha}^{\mu^Y_{\kappa_\alpha}}.
\end{equation}
\end{lemma}

\begin{proof}
  This follows in the same manner as in \cite[Proof of Lemma~4.4]{FZ}, the only difference being in applying (\ref{repr-th1}) of the present paper in place of the unjustified Eq.~(3.32) in \cite{FZ} (see Remark~\ref{compare} above for details).
\end{proof}

As shown in \cite[Theorems~4.9, 4.11]{FZ}, the kernel $g_\alpha$ is {\it perfect}. Furthermore, it satisfies the Frostman and domination maximum principles in the following sense.

\begin{theorem}\label{gFDom}Given $\mu,\nu\in\breve{\mathfrak M}^+(D)$, the following {\rm(a)} and {\rm(b)} are fulfilled.
\begin{itemize}
  \item[{\rm(a)}] If $U^\mu_{g_\alpha}\leqslant1$ $\mu$-a.e., then the same inequality holds true on all of $\mathbb R^n$.
  \item[{\rm(b)}] Assume $\mu$ is $c_{g_\alpha}$-absolutely continuous.\footnote{For extendible measures, the concepts of $c_{g_\alpha}$- and $c_{\kappa_\alpha}$-absolute continuity coincide. Indeed,
      \[c_{g_\alpha}(K)=0\iff c_{\kappa_\alpha}(K)=0\quad\text{for any compact $K\subset D$},\] which follows by noting that  $U_{\kappa_\alpha}^{(\varepsilon_x)_{\kappa_\alpha}^Y}(y)$, $x\in D$, is bounded when $y$ ranges over compact $K\subset D$.} If moreover
      \[U_{g_\alpha}^\mu\leqslant U_{g_\alpha}^\nu\quad\text{$\mu$-a.e.,}\] then the same inequality holds true on all of $D$.
\end{itemize}
\end{theorem}

\begin{proof}
  (a) Applying Lemma~\ref{gpot} gives $U^\mu_{\kappa_\alpha}\leqslant1+U_{\kappa_\alpha}^{\mu^Y_{\kappa_\alpha}}$ $\mu$-a.e., whence the claim in consequence of \cite[Theorem~1.V.10(b)]{Doob} if $\alpha=2$ or \cite[Theorem~1.29]{L} otherwise.

  (b) For $\alpha=2$, see \cite[Theorem~1.V.10(b)]{Doob}. For $\alpha<2$, this follows by a slight modification of \cite[Proof of Theorem~4.6]{FZ}.
\end{proof}

\subsection{$\alpha$-Green balayage}\label{ssec5} In what follows, fix a proper, relatively closed subset $F$ of the domain $D$ with $c_{g_\alpha}(F)>0$, and denote $\Omega:=D\setminus F$. Define
\[F^r:=(F\cup Y)^r\cap D.\]

\begin{theorem}\label{bal-exx}
 For any $\mu\in\breve{\mathfrak M}^+(D)$ such that $\mu|_{F^r}$ is $c_{\kappa_\alpha}$-absolutely continuous, there exists the unique $c_{\kappa_\alpha}$-absolutely continuous measure $\mu^F_{g_\alpha}\in\breve{\mathfrak M}^+(F;D)$ with
 \begin{equation}\label{bal1'}
 U_{g_\alpha}^{\mu^F_{g_\alpha}}(y)=U_{g_\alpha}^\mu(y)\quad\text{for all $y\in F^r$}\quad\text{{\rm(}hence, q.e.\ on $F${\rm)}};
 \end{equation}
 this $\mu^F_{g_\alpha}$ is said to be the $g_\alpha$-balayage of $\mu$ onto $F$.
 Actually,
 \begin{equation}\label{bal2}
 \mu^F_{g_\alpha}=\mu^{F\cup Y}_{\kappa_\alpha}\bigl|_F.
 \end{equation}
 If moreover $\mu\in\mathcal E^+_{g_\alpha}$, then the same $\mu^F_{g_\alpha}$ can alternatively be found as the only measure in the cone $\mathcal E^+_{g_\alpha}(F)$ such that
 \begin{equation*}
 \|\mu-\mu^F_{g_\alpha}\|_{g_\alpha}=\min_{\nu\in\mathcal E^+_{g_\alpha}(F)}\,\|\mu-\nu\|_{g_\alpha}.
 \end{equation*}
 \end{theorem}

 \begin{proof}
   See \cite[Theorem~2.1]{Z-AMP}. For $\mu:=\varepsilon_x$, $x\in\Omega$, cf.\ Frostman \cite[Section~5]{Fr}.
 \end{proof}

 It follows from (\ref{bal1'}) by utilizing Theorem~\ref{gFDom}(b) that
 \begin{equation*}
 U_{g_\alpha}^{\mu^F_{g_\alpha}}\leqslant U_{g_\alpha}^\mu\quad\text{on all of $D$};
 \end{equation*}
 whence, by the principle of positivity of mass for $g_\alpha$-potentials \cite[Theorem~4.13]{FZ},
\begin{equation}\label{bal4g}
\mu_{g_\alpha}^F(D)\leqslant\mu(D).
\end{equation}

\subsection{When does equality in (\ref{bal4g}) hold?} This question is answered by the following slight improvement of \cite[Theorem~3.7]{Z-AMP}.

\begin{theorem}\label{th-balM2} Unless $\alpha<2$, assume that $\Omega$ is connected. Fix $\mu\in\breve{\mathfrak M}^+(D)$ such that $\mu|_F$ is $c_{\kappa_\alpha}$-absolutely continuous while $\mu|_\Omega\ne0$.
Then {\rm(i$_1$)}--{\rm(iii$_1$)} are equivalent.
\begin{itemize}
\item[{\rm(i$_1$)}] $\mu^F_{g_\alpha}(D)=\mu(D)$.
\item[{\rm(ii$_1$)}] $\omega_\alpha(x,\overline{\mathbb R^n}\setminus D;\Omega)=0$ \ $\mu$-a.e.\ on\/ $\Omega$.
\item[{\rm(iii$_1$)}] $\Omega^c$ is not $\alpha$-thin at infinity, and
\begin{equation*}
\omega_\alpha(x,Y;\Omega)=0\quad\text{$\mu$-a.e.\ on\/ $\Omega$.}
\end{equation*}
\end{itemize}
\end{theorem}

\begin{proof}
This follows in a manner similar to that in \cite[Proof of Theorem~3.7]{Z-AMP}.
 \end{proof}

\section{The Gauss variational problem for $g_\alpha$-potentials}\label{sec2}

In what follows, we keep all the conventions introduced in the preceding Section~\ref{sec1}.

\subsection{Statement of the problem}\label{sec-st} In the rest of this paper, we fix a bounded (hence, extendible) measure $\vartheta\in\breve{\mathfrak M}^+(\Omega;D)$, $\vartheta\ne0$, such that
\begin{equation}\label{dist}
\varrho:={\rm dist}\bigl(S(\vartheta;D),F\bigr):=\inf_{(x,y)\in S(\vartheta;D)\times F}\,|x-y|>0.
\end{equation}
Treating $\vartheta$ as a charge creating the {\it external field}
\begin{equation}\label{field}
f:=-U^\vartheta_{g_\alpha},
\end{equation}
we are interested in the problem on minimizing {\it the Gauss functional} $I_{g_\alpha,f}(\mu)$,\footnote{In constructive function theory, $I_{g_\alpha,f}(\cdot)$ is also referred to as the {\it $f$-weighted $\alpha$-Green energy}. For the terminology used here, see e.g.\ \cite{L,O,ST,Z-Oh}.}
\begin{equation}\label{G}
 I_{g_\alpha,f}(\mu):=I_{g_\alpha}(\mu)+2\int f\,d\mu=\|\mu\|^2_{g_\alpha}-2\int U^\vartheta_{g_\alpha}\,d\mu,
\end{equation}
$\mu$ ranging over the class $\mathcal E_{g_\alpha}^+(F,1)$. That is, {\it does there exist $\lambda_{F,f}\in\mathcal E_{g_\alpha}^+(F,1)$ with}
\begin{equation}\label{sol}
 I_{g_\alpha,f}(\lambda_{F,f})=\inf_{\mu\in\mathcal E_{g_\alpha}^+(F,1)}\,I_{g_\alpha,f}(\mu)=:w_{g_\alpha,f}(F)?
\end{equation}
Since $\mathcal E_{g_\alpha}^+(F,1)\ne\varnothing$ because of $c_{g_\alpha}(F)>0$ (cf.\ (\ref{ifff})), problem (\ref{sol}) makes sense.

In view of (\ref{g}) and (\ref{dist}),
\begin{equation}\label{Ifin}I_{g_\alpha}(\vartheta,\mu)\leqslant I_{\kappa_\alpha}(\vartheta,\mu)\leqslant\vartheta(D)/\varrho^{n-\alpha}=:M<\infty\quad\text{for all $\mu\in\mathcal E_{g_\alpha}^+(F,1)$},\end{equation}
whence, by virtue of (\ref{G}) and the strict positive definiteness of the kernel $g_\alpha$,
\begin{equation}\label{wfin}
-\infty<-2M\leqslant w_{g_\alpha,f}(F)<\infty.
\end{equation}
This enables us to show, by means of standard arguments based on the convexity of the class $\mathcal E_{g_\alpha}^+(F,1)$, the strict positive definiteness of the kernel $g_\alpha$, and the parallelogram identity in the pre-Hilbert space $\mathcal E_{g_\alpha}$, that the solution $\lambda_{F,f}$ to problem (\ref{sol}) is {\it unique} (if it exists). See e.g.\ \cite[Lemma~6]{Z5a}.

By (\ref{dist}) and (\ref{field}), $f|_F$ is {\it continuous}.\footnote{When speaking of a continuous function, we generally understand that the values are {\it finite} real numbers.} Thus, if $F=:K$ is compact, then $\int f\,d\mu$ is vaguely continuous on $\mathfrak M^+(K;D)$, and therefore, by the principle of descent \cite[Lemma~2.2.1(e)]{F1}, the Gauss functional $I_{g_\alpha,f}(\cdot)$ is vaguely l.s.c.\ on $\mathfrak M^+(K;D)$. Since the class $\mathcal E_{g_\alpha}^+(K,1)$ is vaguely compact \cite[Section~III.1, Corollary~3 to Proposition~15]{B2}, {\it the existence of $\lambda_{K,f}$ immediately follows}. But if $F$ is noncompact, then these arguments, based on the vague topology only, fail down, and the problem on the existence of the solution $\lambda_{F,f}$ to problem (\ref{sol}) becomes "rather difficult" (Ohtsuka \cite[p.~219]{O}).

\subsection{Main results}\label{sec-main} Problem (\ref{sol}) will be analyzed below in the framework of the approach suggested in our recent paper \cite{Z-Oh}, which is based on the perfectness of the kernel in question, and hence on the simultaneous use of both the strong and vague topologies on the pre-Hilbert space $\mathcal E_{g_\alpha}$ (see Sections~\ref{ssec1}, \ref{ssec4} above). The theories of $\alpha$-Green balayage and $\alpha$-Green equilibrium measures are also particularly helpful.

Necessary and/or sufficient conditions for the existence of the solution $\lambda_{F,f}$ to problem (\ref{sol}) are obtained in Theorems~\ref{sol-1}--\ref{sol-infin}. We also give alternative characterizations of $\lambda_{F,f}$ (Theorems~\ref{sol-1}--\ref{sol-infin}), and we prove assertions on convergence when $F$ is approximated by partially ordered families of sets (Theorems~\ref{continuity}, \ref{continuity2}). Furthermore, we establish a description of the support $S(\lambda_{F,f};D)$ (Theorems~\ref{sup-desc}, \ref{th-ex}), thereby answering the question raised by Ohtsuka in \cite[p.~284, Open question~2.1]{O}.\footnote{It is worth noting that in the case $c_{\kappa_\alpha}(Y)=0$, problem (\ref{sol}) is reduced to that on minimizing $I_{\kappa_\alpha,f}(\mu):=\|\mu\|^2_{\kappa_\alpha}-2\int U^\vartheta_{\kappa_\alpha}\,d\mu$ over the class $\mathcal E^+_{\kappa_\alpha}(F\cup Y,1)$, investigated in details in \cite{Z-Rarx,Z-CA25}.}

To present those results, we first recall the following well-known theorem, providing characteristic properties of the solution $\lambda_{F,f}$. It can be derived from our earlier paper
\cite{Z5a} (see Theorems~1, 2 and Proposition~1 therein), dealing with an arbitrary positive definite kernel on a locally compact space.

\begin{theorem}\label{th-ch2}For $\lambda\in\mathcal E_{g_\alpha}^+(F,1)$ to serve as the {\rm(}unique{\rm)} solution
$\lambda_{F,f}$ to problem~{\rm(\ref{sol})}, it is necessary and sufficient that either of the two inequalities holds
\begin{align}\label{1}U_{g_\alpha,f}^\lambda&\geqslant c_{F,f}\quad\text{q.e.\ on $F$},\\
U_{g_\alpha,f}^\lambda&\leqslant c_{F,f}\quad\text{$\lambda$-a.e.\ on $D$,}\label{2}\end{align}
where
\[U_{g_\alpha,f}^\lambda:=U_{g_\alpha}^\lambda+f\]
is said to be the $f$-weighted potential of $\lambda$, while
\begin{equation}\label{cc}
c_{F,f}:=\int U_{g_\alpha,f}^{\lambda}\,d\lambda\in(-\infty,\infty)
\end{equation}
is referred to as the $f$-weighted equilibrium constant.
\end{theorem}

Thus, if the solution $\lambda_{F,f}$ to problem (\ref{sol}) exists, then necessarily
\[\lambda_{F,f}\in\Lambda_{F,f},\]
where
\begin{equation}\label{La}\Lambda_{F,f}:=\bigl\{\mu\in\mathcal E^+_{g_\alpha}(D):\ U_{g_\alpha,f}^\mu\geqslant c_{F,f}\quad\text{q.e.\ on $F$}\bigr\}.\end{equation}
(Note that for any $\nu\in\mathcal E_{g_\alpha}$, $U^\nu_{g_\alpha}$ is well defined and finite q.e.\ on $D$, cf.\ \cite[Corollary to Lemma~3.2.3]{F1}, the kernel $g_\alpha$ being strictly positive definite; and hence so is $U^\nu_{g_\alpha,f}$.)

We are now in a position to formulate the main results of the current paper.

\begin{theorem}\label{sol-1}If $\vartheta^F_{g_\alpha}(F)=1$,
then the solution $\lambda_{F,f}$ to problem {\rm(\ref{sol})} does exist. Furthermore, then
\begin{equation}\label{eq-sol-1}
\lambda_{F,f}=\vartheta^F_{g_\alpha},\quad c_{F,f}=0,\quad w_{g_\alpha,f}(F)=I_{g_\alpha,f}(\vartheta^F_{g_\alpha})=-I_{g_\alpha}(\vartheta^F_{g_\alpha})\in(-\infty,0),
\end{equation}
while $\lambda_{F,f}$ can alternatively be characterized by any one of the following {\rm(i$_2$)--(iii$_2$)}:
\begin{itemize}
\item[{\rm(i$_2$)}] $\lambda_{F,f}$ is the unique measure in the class $\mathcal E_{g_\alpha}^+(F)$ such that
\begin{equation*}
U^{\lambda_{F,f}}_{g_\alpha,f}=c_{F,f}\quad\text{q.e.\ on $F$}.
\end{equation*}
\item[{\rm(ii$_2$)}] $\lambda_{F,f}$ is the unique measure in the class $\Lambda_{F,f}$ such that
\[U^{\lambda_{F,f}}_{g_\alpha}=\min_{\mu\in\Lambda_{F,f}}\,U^\mu_{g_\alpha}\quad\text{q.e.\ on $D$},\]
$\Lambda_{F,f}$ being introduced by means of {\rm(\ref{La})}.
\item[{\rm(iii$_2$)}] $\lambda_{F,f}$ is the unique measure in the class $\Lambda_{F,f}$ such that
\[\|\lambda_{F,f}\|_{g_\alpha}=\min_{\mu\in\Lambda_{F,f}}\,\|\mu\|_{g_\alpha}.\]
\end{itemize}
\end{theorem}

\begin{theorem}\label{sol-fin}
If $c_{g_\alpha}(F)<\infty$, then the solution $\lambda_{F,f}$ to problem {\rm(\ref{sol})} does exist. Assume in addition that $\vartheta^F_{g_\alpha}(F)\leqslant1$. Then this $\lambda_{F,f}$ can be written in the form\footnote{Here $\gamma_{F,g_\alpha}$ denotes the $g_\alpha$-equilibrium measure on $F$, normalized by $\gamma_{F,g_\alpha}(F)=c_{g_\alpha}(F)$, cf.\ Section~\ref{ssec1}. Also note that $\vartheta^F_{g_\alpha}(F)\leqslant1$ necessarily holds if $\vartheta(D)\leqslant1$, cf.\ (\ref{bal4g}).}
\begin{equation}\label{repr}
\lambda_{F,f}=\vartheta^F_{g_\alpha}+c_{F,f}\gamma_{F,g_\alpha},
\end{equation}
where the $f$-weighted equilibrium constant $c_{F,f}$ admits the representation
\begin{equation}\label{eta}
c_{F,f}=\frac{1-\vartheta^F_{g_\alpha}(F)}{c_{g_\alpha}(F)}\in[0,\infty),
\end{equation}
and moreover $\lambda_{F,f}$ is uniquely characterized by any one of the above {\rm(i$_2$)--(iii$_2$)}.
\end{theorem}

\begin{theorem}\label{sol-infin}Unless $\alpha<2$, assume that $\Omega$ is connected. If moreover
\begin{equation}\label{necsuf}
\omega_\alpha(x,\overline{\mathbb R^n}\setminus D;\Omega)=0\quad\text{on all of \ $\Omega$},
\end{equation}
then
\begin{equation*}\lambda_{F,f}\text{\ exists}\iff\vartheta(D)\geqslant1,\end{equation*}
or equivalently
\begin{equation*}\lambda_{F,f}\text{\ exists}\iff\vartheta^F_{g_\alpha}(D)\geqslant1.\end{equation*}
In the particular case $\vartheta(D)=1$, we actually have $\lambda_{F,f}=\vartheta^F_{g_\alpha}$,\footnote{This fails to hold if $\vartheta(D)=1$ is replaced by $\vartheta(D)>1$, for then, by Theorem~\ref{th-balM2}, $\vartheta^F_{g_\alpha}(D)>1$, whence $\lambda_{F,f}\ne\vartheta^F_{g_\alpha}$.} and so Theorem~{\rm\ref{sol-1}} is fully applicable.
\end{theorem}

Let $\check{F}$ be the {\it reduced kernel} of $F$, defined as the set of all $x\in F$ such that
$c_{g_\alpha}(F\cap U_x)>0$ for any open neighborhood $U_x$ of $x$ in $D$ (cf.\ \cite[p.~164]{L}). By virtue of the countable subadditivity of $c_{g_\alpha}(\cdot)$ on Borel subsets of $D$, cf.\ \cite[Lemma~2.3.5]{F1}, we have $c_{g_\alpha}(F\setminus\check{F})=0$. It is also easy to see that the reduced kernel of a relatively closed subset of $D$ is likewise relatively closed.

\begin{theorem}\label{sup-desc}Unless $\alpha<2$, assume $\Omega$ is connected. Also assume that either
\begin{equation}\label{C1}
\vartheta^F_{g_\alpha}(F)=1,
\end{equation}
or
\begin{equation}\label{C2}
c_{g_\alpha}(F)<\infty\quad\text{and}\quad\vartheta^F_{g_\alpha}(F)\leqslant1.
\end{equation}
Then
\begin{equation}\label{RRR}S(\lambda_{F,f};D)=\left\{
\begin{array}{cl}\check{F}&\text{if \ $\alpha<2$},\\
\partial_D\check{F}&\text{otherwise}.\\ \end{array} \right.
\end{equation}
\end{theorem}

We denote by $\mathfrak C_F$ the upward directed set of all compact subsets $K$ of $F$, where $K_1\leqslant K_2$ if and only if $K_1\subset K_2$. If a net $(x_K)_{K\in\mathfrak C_F}\subset Y$ converges to $x_0\in Y$, $Y$ being a topological space, then we shall indicate this fact by writing
\begin{equation*}x_K\to x_0\quad\text{in $Y$ as $K\uparrow F$}.\end{equation*}
In particular, it follows from (\ref{minS10}) and (\ref{minS1}) (see below) that\footnote{Relation (\ref{wdown}) remains valid if the net $(K)_{K\in\mathfrak C_F}$ is replaced by an increasing sequence $(F_j)$ of relatively closed subsets of $D$ with the union $F$. See also Remark~\ref{for-f}.}
\begin{equation}\label{wdown}w_{g_\alpha,f}(K)\downarrow w_{g_\alpha,f}(F)\quad\text{in $\mathbb R$ as $K\uparrow F$}.\end{equation}

\begin{theorem}\label{continuity}Assume that $\lambda_{F,f}$ exists.\footnote{See Theorems~\ref{sol-1}--\ref{sol-infin} for sufficient conditions for this to occur.} If $K\uparrow F$, then
\begin{equation}\label{Cont}
\lambda_{K,f}\to\lambda_{F,f}\quad\text{strongly and vaguely in $\mathcal E^+_{g_\alpha}$},\end{equation}
whence
\begin{equation}\label{Cont3}c_{K,f}\to c_{F,f},\end{equation}
and also there is a subsequence $(\lambda_{K_j,f})$ of the net $(\lambda_{K,f})_{K\in\mathfrak C_F}$ such that
\begin{equation}\label{Cont2}U^{\lambda_{K_j,f}}_{g_\alpha}\to U^{\lambda_{F,f}}_{g_\alpha}\quad\text{pointwise q.e.\ on $D$ as $j\to\infty$}.\end{equation}
If moreover $\vartheta^F_{g_\alpha}(D)\leqslant1$, then the limit relation {\rm(\ref{Cont3})} can be refined as follows:
\begin{equation}\label{Cont3'}c_{K,f}\downarrow c_{F,f}\quad\text{as $K\uparrow F$}.\end{equation}
\end{theorem}

\begin{theorem}\label{continuity2}Let $(F_s)_{s\in S}$ be a decreasing net of relatively closed $F_s\subset D$ with the given intersection $F$, and such that for some $s_1\in S$, $c_{g_\alpha}(F_{s_1})\in(0,\infty)$ and ${\rm dist}\bigl(S(\vartheta;D),F_{s_1}\bigr)>0$. Then
\begin{equation}\label{wup}w_{g_\alpha,f}(F_s)\uparrow w_{g_\alpha,f}(F)\quad\text{as $s$ ranges through $S$},\end{equation}
and moreover {\rm(\ref{Cont})--(\ref{Cont2})} hold true for $(F_s)_{s\in S}$ in place of $(K)_{K\in\mathfrak C_F}$. If, in addition,  $\vartheta^{F_{s_1}}_{g_\alpha}(D)\leqslant1$, then also
\begin{equation*}c_{F_s,f}\uparrow c_{F,f}\quad\text{as $s$ ranges through $S$}.\end{equation*}
\end{theorem}

\begin{remark}\label{for-f}
Theorem~\ref{continuity} remains valid if the net $(K)_{K\in\mathfrak C_F}$ is replaced by an increasing {\it sequence} of relatively closed sets $F_k\subset D$ with the union $F$ and such that the minimizers $\lambda_{F_k,f}$ exist. This can be seen in a manner similar to that in the proof of Theorem~\ref{continuity} (Section~\ref{pr-cont}), the only difference being in applying the monotone convergence theorem \cite[Section~IV.1, Theorem~3]{B2} in place of \cite[Lemma~1.2.2]{F1} (see the proof of the preparatory Lemma~\ref{l-pot} in Section~\ref{sec-extr}).
\end{remark}

\begin{example}\label{ex1} Let $c_{\kappa_\alpha}(Y)=0$, and let $F\subset D$ be not $\alpha$-thin at infinity. Also assume that either $\alpha<2$, or $\Omega$ is connected.

\begin{theorem}\label{th-ex}
For these particular $D$ and $F$, the solution $\lambda_{F,f}$ to problem {\rm(\ref{sol})} exists if and only if $\vartheta(D)\geqslant1$. Besides, if $\vartheta(D)=1$, then $S(\lambda_{F,f};D)$ is given by means of formula {\rm(\ref{RRR})}, while otherwise $S(\lambda_{F,f};D)$ is a compact subset of $F$.
\end{theorem}

\end{example}

For the proofs of Theorems~\ref{sol-1}--\ref{continuity2} and \ref{th-ex}, see Sections~\ref{pr-sol-1}--\ref{th-ex-pr}.

\section{Preparatory assertions}

To verify the above theorems, we first need to establish some auxiliary results.

\begin{lemma}\label{capinf}
Unless $\alpha<2$, assume $\Omega$ is connected. If there exists $x_0\in\Omega$ with
\begin{equation}\label{x0}\omega_\alpha(x_0,\overline{\mathbb R^n}\setminus D;\Omega)=0,\end{equation}
then
\[c_{g_\alpha}(F)=\infty.\]
\end{lemma}

\begin{proof} Assuming to the contrary that $c_{g_\alpha}(F)<\infty$, we conclude from the perfectness of the kernel $g_\alpha$  (Section~\ref{ssec4}) that there exists the (unique) $g_\alpha$-equilibrium measure $\gamma_{F,g_\alpha}$ (cf.\ \cite[Theorem~4.12]{FZ}), and moreover
\begin{equation*}
U^{\gamma_F}_{g_\alpha}<1\quad\text{on all of $\Omega$}
\end{equation*}
(for more details, see \cite[Lemma~6.5]{Z-AMP}). Therefore, by virtue of \cite[Lemma~6.6]{Z-AMP} with $\mu:=\varepsilon_x$, where $x\in\Omega$ is arbitrary,
\[(\varepsilon_x)^F_{g_\alpha}(D)=\int U^{\gamma_F}_{g_\alpha}\,d\varepsilon_x=U^{\gamma_F}_{g_\alpha}(x)<1,\]
whence
\begin{equation}\label{Str}
(\varepsilon_x)^F_{g_\alpha}(D)<\varepsilon_x(D)\quad\text{for all $x\in\Omega$}.
\end{equation}
Applying Theorem~\ref{th-balM2} to $\mu:=\varepsilon_{x_0}$, where $x_0\in\Omega$ meets (\ref{x0}), we thus arrive at a contradiction with (\ref{Str}).
\end{proof}

\begin{lemma}\label{str-cont}
The Gauss functional $I_{g_\alpha,f}(\cdot)$ is strongly continuous on the strongly complete cone $\mathcal E^+_{g_\alpha}(F)$. Thus, if a net $(\mu_s)\subset\mathcal E_{g_\alpha}^+(F)$ is strongly Cauchy, then $(\mu_s)$ converges strongly {\rm(}hence, vaguely{\rm)} to some unique $\mu_0\in\mathcal E_{g_\alpha}^+(F)$, and moreover
\begin{equation}\label{eq-str-cont}
\lim_{s}\,I_{g_\alpha,f}(\mu_s)=I_{g_\alpha,f}(\mu_0).
\end{equation}
\end{lemma}

\begin{proof}
As the kernel $g_\alpha$ is perfect (see Section~\ref{ssec1} for details), the cone $\mathcal E_{g_\alpha}^+$ is strongly complete, and moreover the strong topology on $\mathcal E_{g_\alpha}^+$ is finer than the vague topology. Being vaguely (hence strongly) closed subcone of the (strongly complete) cone $\mathcal E_{g_\alpha}^+$, cf.\ \cite[Section~III.2, Proposition~6]{B2}, $\mathcal E_{g_\alpha}^+(F)$ is likewise strongly complete, and therefore a strongly Cauchy net $(\mu_s)\subset\mathcal E_{g_\alpha}^+(F)$ converges strongly (hence, also vaguely) to some unique $\mu_0\in\mathcal E_{g_\alpha}^+(F)$, the strong topology on $\mathcal E_{g_\alpha}(D)$ as well as the vague topology on $\mathfrak M(D)$ being Hausdorff. It thus remains to verify the limit relation (\ref{eq-str-cont}).

In a manner similar to that in (\ref{Ifin}), we obtain, by use of (\ref{bal1'}) and (\ref{bal4g}),
\[I_{g_\alpha}(\vartheta^F_{g_\alpha})=\int U_{g_\alpha}^{\vartheta^F_{g_\alpha}}\,d\vartheta^F_{g_\alpha}=
\int U_{g_\alpha}^\vartheta\,d\vartheta^F_{g_\alpha}\leqslant\int U_{\kappa_\alpha}^\vartheta\,d\vartheta^F_{g_\alpha}\leqslant\vartheta(D)^2/\varrho^{n-\alpha}<\infty,\]
whence
\begin{equation}\label{E}
 \vartheta^F_{g_\alpha}\in\mathcal E_{g_\alpha}^+(F).
\end{equation}
Noting from (\ref{bal1'}) that for any $\mu\in\mathcal E_{g_\alpha}^+(F)$, $U_{g_\alpha}^\vartheta=U_{g_\alpha}^{\vartheta^F_{g_\alpha}}$ $\mu$-a.e., and so
\begin{equation}\label{reprR}\int U_{g_\alpha}^\vartheta\,d\mu=\int U_{g_\alpha}^{\vartheta^F_{g_\alpha}}\,d\mu,\end{equation}
we thus get
\begin{equation}\label{reprr}
I_{g_\alpha,f}(\mu)=\|\mu\|_{g_\alpha}^2-2\bigl\langle\vartheta^F_{g_\alpha},\mu\bigr\rangle_{g_\alpha}=
\|\mu-\vartheta^F_{g_\alpha}\|_{g_\alpha}^2-\|\vartheta^F_{g_\alpha}\|_{g_\alpha}^2,
\end{equation}
and the strong continuity of $I_{g_\alpha,f}(\cdot)$ on $\mathcal E^+_{g_\alpha}(F)$ follows.
\end{proof}

\begin{corollary}\label{Cor} We have
\begin{equation}\label{obs}
w_{g_\alpha,f}(F)\geqslant-I_{g_\alpha}(\vartheta^F_{g_\alpha})=I_{g_\alpha,f}(\vartheta^F_{g_\alpha}).
\end{equation}
\end{corollary}

\begin{proof}
In view of the strict positive definiteness of $g_\alpha$, this follows from (\ref{reprr}).
\end{proof}

\subsection{A dual extremal problem} Parallel with $f=-U_{g_\alpha}^\vartheta$, we shall consider the dual external field $\tilde{f}$, given by
\begin{equation}\label{field2}
\tilde{f}:=-U_{g_\alpha}^{\vartheta^F_{g_\alpha}}.
\end{equation}
It is seen from (\ref{reprR}) that
\begin{equation*}I_{g_\alpha,f}(\mu)=I_{g_\alpha,\tilde{f}}(\mu)\quad\text{for all $\mu\in\mathcal E_{g_\alpha}^+(F)$},\end{equation*}
whence
\begin{equation}\label{dual}
w_{g_\alpha,f}(F)=w_{g_\alpha,\tilde{f}}(F):=\inf_{\mu\in\mathcal E_{g_\alpha}^+(F,1)}\,I_{g_\alpha,\tilde{f}}(\mu).
\end{equation}
Furthermore, the (original) problem (\ref{sol}) is solvable if and only if so is the (dual) problem (\ref{dual}), and in the affirmative case
\begin{equation}\label{duall}
\lambda_{F,f}=\lambda_{F,\tilde{f}},\quad c_{F,f}=c_{F,\tilde{f}}.
\end{equation}

An advantage of the dual problem (\ref{dual}) if compared with the original problem (\ref{sol}) is that $\vartheta^F_{g_\alpha}$ is of finite $g_\alpha$-energy, cf.\ (\ref{E}), and so \cite{Z-Oh} is fully applicable.

\subsection{Extremal measures}\label{sec-extr} A net $(\mu_s)\subset\mathcal E_{g_\alpha}^+(F,1)$ is said to be {\it minimizing} in problem (\ref{sol}) if
\begin{equation}\label{minn}
\lim_{s}\,I_{g_\alpha,f}(\mu_s)=w_{g_\alpha,f}(F);
\end{equation}
let $\mathbb M_{g_\alpha,f}(F)$ consist of all those $(\mu_s)$. It is clear from (\ref{wfin}) that $\mathbb M_{g_\alpha,f}(F)\ne\varnothing$.

\begin{lemma}\label{l-extr}
There exists a unique $\xi_{F,f}\in\mathcal E_{g_\alpha}^+(F)$ called the extremal measure in problem {\rm(\ref{sol})} and such that for every minimizing net $(\mu_s)\in\mathbb M_{g_\alpha,f}(F)$,
\begin{equation}\label{m-extr}
\mu_s\to\xi_{F,f}\quad\text{strongly and vaguely in $\mathcal E_{g_\alpha}^+(F)$}.
\end{equation}
This yields
\begin{equation}\label{ext-eq1}
I_{g_\alpha,f}(\xi_{F,f})=w_{g_\alpha,f}(F).
\end{equation}
\end{lemma}

\begin{proof} It follows by use of standard arguments, based on the convexity of the cone $\mathcal E_{g_\alpha}^+(F)$, the parallelogram identity in the pre-Hil\-bert space $\mathcal E_{g_\alpha}$, and the strict positive definiteness of the kernel $g_\alpha$, that any $(\mu_s)\in\mathbb M_{g_\alpha,f}(F)$ is strongly Cauchy (cf.\ \cite[Proof of Lemma~4.1]{Z-Oh}). This yields (\ref{m-extr}), the cone $\mathcal E^+_{g_\alpha}(F)$ being strongly complete in view of the perfectness of the kernel $g_\alpha$. Since $I_{g_\alpha,f}(\cdot)$ is strongly continuous on $\mathcal E^+_{g_\alpha}(F)$ (Lemma~\ref{str-cont}), (\ref{ext-eq1}) is deduced from (\ref{m-extr}) by substituting (\ref{eq-str-cont}) into (\ref{minn}).\end{proof}

Noting that the mapping $\mu\mapsto\mu(D)$ is vaguely l.s.c.\ on $\mathfrak M^+(D)$ \cite[Section~IV.1, Proposition~4]{B2}, we infer from (\ref{m-extr}) that,
in general,
\begin{equation}\label{ext-eq4}
\xi_{F,f}(D)\leqslant1.
\end{equation}

\begin{corollary}\label{l-extr5}The solution $\lambda_{F,f}$ to problem {\rm(\ref{sol})} exists if and only if equality prevails in {\rm(\ref{ext-eq4})}, i.e.
\begin{equation}\label{ext-eq2}
\xi_{F,f}(D)=1,
\end{equation}
and in the affirmative case,
\begin{equation}\label{ext-eq3}
\xi_{F,f}=\lambda_{F,f}.
\end{equation}
\end{corollary}

\begin{proof}
If (\ref{ext-eq2}) holds, then combining it with both $\xi_{F,f}\in\mathcal E_{g_\alpha}^+(F)$ and (\ref{ext-eq1}) gives (\ref{ext-eq3}). For the opposite, assume $\lambda_{F,f}$ exists. Since the trivial sequence $(\lambda_{F,f})$ is obviously minimizing, it must converge strongly to each of $\lambda_{F,f}$ and $\xi_{F,f}$ (Lemma~\ref{l-extr}), which immediately results in (\ref{ext-eq3}), the strong topology on $\mathcal E_{g_\alpha}$ being Hausdorff.
\end{proof}

\begin{lemma}\label{l-pot}
For the extremal measure $\xi:=\xi_{F,f}$, we have
\begin{equation}\label{e-pot1}
U^\xi_{g_\alpha,f}\geqslant C_\xi\quad\text{q.e.\ on $F$},
\end{equation}
where
\begin{equation}\label{Cxi}
C_\xi:=\int U^\xi_{g_\alpha,f}\,d\xi\in(-\infty,\infty).
\end{equation}
\end{lemma}

\begin{proof} As noted in Section~\ref{sec-st}, problem (\ref{sol}) is (uniquely) solvable for every $K\in\mathfrak C_F$. Our first aim is to show that those solutions form a minimizing net, i.e.
\begin{equation}\label{min-net}
(\lambda_{K,f})_{K\in\mathfrak C_F}\in\mathbb M_{g_\alpha,f}(F).
\end{equation}
Since $\mathcal E^+_{g_\alpha}(K_1,1)\subset\mathcal E^+_{g_\alpha}(K_2,1)\subset\mathcal E^+_{g_\alpha}(F,1)$ for any $K_1,K_2\in\mathfrak C_F$ such that $K_1\leqslant K_2$, the net $\bigl(w_{g_\alpha,f}(K)\bigr)_{K\in\mathfrak C_F}$ decreases, and moreover
\begin{equation}\label{minS10}\lim_{K\in\mathfrak C_F}\,w_{g_\alpha,f}(K)=\lim_{K\in\mathfrak C_F}\,I_{g_\alpha,f}(\lambda_{K,f})\geqslant w_{g_\alpha,f}(F).\end{equation}
We are thus reduced to showing that
\begin{equation}\label{minS1}
w_{g_\alpha,f}(F)\geqslant\lim_{K\in\mathfrak C_F}\,w_{g_\alpha,f}(K).
\end{equation}
Applying \cite[Lemma~1.2.2]{F1} to each of the positive, l.s.c., $\mu$-int\-eg\-r\-ab\-le functions $U_{g_\alpha}^\mu$ and $U_{g_\alpha}^\vartheta$, we see that for every $\mu\in\mathcal E^+_{g_\alpha}(F,1)$,
\begin{equation}\label{minS2}I_{g_\alpha,f}(\mu)\geqslant\lim_{K\uparrow A}\,I_{g_\alpha,f}(\mu|_K).\end{equation}
(Regarding the $\mu$-int\-eg\-r\-ab\-ility of $U_{g_\alpha}^\vartheta$, see (\ref{reprR}).) Noting that $\mu(K)\uparrow\mu(F)=1$ as $K\uparrow F$, and letting $\mu$ range over $\mathcal E^+_{g_\alpha}(F,1)$, we get (\ref{minS1}) from (\ref{minS2}), whence (\ref{min-net}).

Thus, according to Lemma~\ref{l-extr},
\begin{equation}\label{ltoxip0}
\lambda_{K,f}\to\xi\quad\text{strongly and vaguely as $K\uparrow A$}.
\end{equation}
The strong topology on $\mathcal E_{g_\alpha}$ being first-countable, there exists a subsequence $(K_j)$ of the net $(K)_{K\in\mathfrak C_A}$ such that
\begin{equation}\label{ltoxip}U^{\lambda_{K_j,f}}_{g_\alpha}\to U^\xi_{g_\alpha}\quad\text{pointwise q.e.\ on $D$ as $j\to\infty$,}\end{equation}
cf.\ the paragraph following Lemma~4.3.3 in \cite{F1}. Now, applying (\ref{1}) and (\ref{cc}) to each of those $\lambda_{K_j,f}$, and then passing to the limits in the inequalities thereby obtained, we deduce (\ref{e-pot1}) from (\ref{ltoxip0}) and (\ref{ltoxip}), for
\begin{align*}\lim_{j\to\infty}\,c_{K_j,f}&=\lim_{j\to\infty}\,\Bigl(\|\lambda_{K_j,f}\|^2_{g_\alpha}-\int U^\vartheta_{g_\alpha}\,d\lambda_{K_j,f}\Bigr)=\lim_{j\to\infty}\,\Bigl(\|\lambda_{K_j,f}\|^2_{g_\alpha}-
\langle\vartheta^F_{g_\alpha},\lambda_{K_j,f}\rangle_{g_\alpha}\Bigr)\\
{}&=\|\xi\|^2_{g_\alpha}-\langle\vartheta^F_{g_\alpha},\xi\rangle_{g_\alpha}=\|\xi\|^2_{g_\alpha}-\int U^\vartheta_{g_\alpha}\,d\xi=\int U^\xi_{g_\alpha,f}\,d\xi=:C_\xi.\end{align*}
(While doing that, we have used the countable subadditivity of the outer $g_\alpha$-capacity, cf.\ \cite[Lemma~2.3.5]{F1}.)
\end{proof}

\section{Proof of Theorem~\ref{sol-1}}\label{pr-sol-1} If $\vartheta^F_{g_\alpha}(F)=1$, then $\vartheta^F_{g_\alpha}\in\mathcal E^+_{g_\alpha}(F,1)$, cf.\ (\ref{E}), and therefore $I_{g_\alpha,f}(\vartheta^F_{g_\alpha})\geqslant w_{g_\alpha,f}(F)$. Combined with (\ref{obs}), this implies the solvability of problem (\ref{sol}) as well as the first and the last relations in (\ref{eq-sol-1}). According to (\ref{bal1'}) with $\mu:=\vartheta$,
\[U_{g_\alpha}^\vartheta=U_{g_\alpha}^{\vartheta^F_{g_\alpha}}\quad\text{q.e.\ on $F$}\quad\text{(hence, $\vartheta^F_{g_\alpha}$-a.e.),}\] and substituting this equality into (\ref{cc}) with $\lambda:=\vartheta^F_{g_\alpha}$ gives $c_{F,f}=0$, whence (\ref{eq-sol-1}).

To verify the characterizations (i$_2$)--(iii$_2$) of the solution $\lambda_{F,f}$, we first show that
\begin{equation}\label{La2}
\Lambda_{F,f}=\Lambda_{F,\tilde{f}},
\end{equation}
$\Lambda_{F,\tilde{f}}$ being defined by (\ref{La}) with $f$ replaced by the dual external field $\tilde{f}$, cf.\ (\ref{field2}). In fact, $f=\tilde{f}$ q.e.\ on $F$, and so for any $\mu\in\mathcal E^+_{g_\alpha}(D)$,
\begin{equation}\label{U}
U^\mu_{g_\alpha,f}=U^\mu_{g_\alpha,\tilde{f}}\quad\text{q.e.\ on $F$}.
\end{equation}
As $c_{F,f}=c_{F,\tilde{f}}$ by (\ref{duall}), this results in (\ref{La2}) by use of the countable subadditivity of the outer $g_\alpha$-capacity, cf.\ \cite[Lemma~2.3.5]{F1}. Since, again by virtue of (\ref{duall}), $\lambda_{F,f}=\lambda_{F,\tilde{f}}$, utilizing \cite[Theorem~1.6, (i) and (ii)]{Z-Oh} with $\kappa:=g_\alpha$ and $f:=-U^{\vartheta_{g_\alpha}^F}_{g_\alpha}$ proves (ii$_2$) and (iii$_2$). (Recall that \cite[Theorem~1.6]{Z-Oh} is applicable here, for $\vartheta_{g_\alpha}^F\in\mathcal E^+_{g_\alpha}$.)

Finally, according to \cite[Theorem~1.6(iii)]{Z-Oh} applied to $\kappa:=g_\alpha$ and $f:=-U^{\vartheta_{g_\alpha}^F}_{g_\alpha}$, $\lambda_{F,\tilde{f}}$ is the unique measure in $\mathcal E^+_{g_\alpha}(D)$ such that $U_{g_\alpha,\tilde{f}}^{\lambda_{F,\tilde{f}}}=c_{F,\tilde{f}}$ q.e.\ on $F$. Therefore, employing (\ref{U}) with $\mu:=\lambda_{F,\tilde{f}}=\lambda_{F,f}$ yields the remaining assertion (i$_2$).

\section{Proof of Theorem~\ref{sol-fin}}\label{pr-sol-fin} As seen from Lemma~\ref{l-extr} and Corollary~\ref{l-extr5}, the solution $\lambda_{F,f}$  exists if and only if $\xi_{F,f}$, the extremal measure in problem (\ref{sol}), has unit total mass. Since
$\xi_{F,f}\in\mathcal E^+_{g_\alpha}(F)$ is the strong limit of any minimizing net $(\mu_s)\subset\mathcal E^+_{g_\alpha}(F,1)$ (Lemma~\ref{l-extr}), whereas $\mathcal E^+_{g_\alpha}(F,1)$ is strongly complete due to the assumption $c_{g_\alpha}(F)<\infty$
(see Theorem~\ref{th-comp}, applied to the perfect kernel $g_\alpha$), the required equality $\xi_{F,f}(F)=1$ follows.

In the rest of this proof, $\vartheta^F_{g_\alpha}(F)\leqslant1$. Then
\begin{equation}\label{leq}
\eta_{F,f}:=\frac{1-\vartheta^F_{g_\alpha}(F)}{c_{g_\alpha}(F)}\in[0,\infty),
\end{equation}
whence
\begin{equation}\label{La3}
\lambda:=\vartheta^F_{g_\alpha}+\eta_{F,f}\gamma_{F,g_\alpha}\in\mathcal E^+_{g_\alpha}(F,1).
\end{equation}
Moreover,
\[U^\lambda_{g_\alpha,f}=U^\lambda_{g_\alpha}+f=
\bigl(U_{g_\alpha}^{\vartheta^F_{g_\alpha}}-U_{g_\alpha}^\vartheta\bigr)+\eta_{F,f}U^{\gamma_{F,g_\alpha}}_{g_\alpha}=
\eta_{F,f}\quad\text{q.e.\ on $F$},\]
the last equality being derived from (\ref{eqpot}) and (\ref{bal1'}) by use of the countable subadditivity of the outer capacity. By virtue of Theorem~\ref{th-ch2}, we thus have
\[\lambda=\lambda_{F,f}\quad\text{and}\quad\eta_{F,f}=c_{F,f},\] which substituted into (\ref{leq}) and (\ref{La3}) proves (\ref{eta}) and (\ref{repr}), respectively.

Noting that under the accepted assumptions $c_{g_\alpha}(F)<\infty$ and $\vartheta^F_{g_\alpha}(F)\leqslant1$, \cite[Theorem~1.6]{Z-Oh} is fully applicable to $\kappa:=g_\alpha$ and $f:=-U^{\vartheta_{g_\alpha}^F}_{g_\alpha}$, we arrive at (i$_2$)--(iii$_2$) in the same manner as in the last two paragraphs of the preceding Section~\ref{pr-sol-1}.

\section{Proof of Theorem~\ref{sol-infin}}\label{pr-sol-infin} Assume that (\ref{necsuf}) holds, and that either $\alpha<2$ or $\Omega:=D\setminus F$ is connected. Then, by virtue of Lemma~\ref{capinf},
\begin{equation}\label{capinfeq}
c_{g_\alpha}(F)=\infty.
\end{equation}
Furthermore, according to Theorem~\ref{th-balM2}, $\vartheta(D)=\vartheta_{g_\alpha}^F(D)$,
and hence, in view of Theorem~\ref{sol-1}, we are reduced to the case where $\vartheta(D)\ne1$, or equivalently
\[\vartheta_{g_\alpha}^F(D)\ne1.\]

Suppose first that $q:=\vartheta_{g_\alpha}^F(D)<1$. The locally compact space $D$ being $\sigma$-com\-pact, there is a sequence of relatively compact, open sets $U_j$ with the union $D$ and such that ${\rm Cl}_D\,U_j\subset U_{j+1}$, see \cite[Section~I.9, Proposition~15]{B1}. Because of (\ref{capinfeq}), $c_{g_\alpha}(F\setminus U_j)=\infty$, $j\in\mathbb N$,\footnote{This follows from the subadditivity of the outer $g_\alpha$-capacity by use of the fact that, due to the strict positive definiteness of the kernel $g_\alpha$, the $g_\alpha$-capacity of any compact set $K\subset D$ is finite.}
and hence one can choose $(\tau_j)\subset\mathcal E^+_{g_\alpha}(F,1)$ such that
\begin{equation}\label{prop}
S(\tau_j;D)\subset D\setminus{\rm Cl}_D\,U_j,\quad\|\tau_j\|_{g_\alpha}<1/j.
\end{equation}
Define
\begin{equation}\label{muj}\mu_j:=\vartheta_{g_\alpha}^F+(1-q)\tau_j.\end{equation}
Clearly, $\mu_j\in\mathcal E^+_{g_\alpha}(F,1)$ for all $j$, and therefore, on account of Corollary~\ref{Cor},
\begin{align*}-\|\vartheta^F_{g_\alpha}\|^2_{g_\alpha}&\leqslant w_{g_\alpha,f}(F)\leqslant\lim_{j\to\infty}\,I_{g_\alpha,f}(\mu_j)=\lim_{j\to\infty}\,\Bigl(\|\mu_j\|^2_{g_\alpha}-2\int U^\vartheta_{g_\alpha}\,d\mu_j\Bigr)\\
{}&=\lim_{j\to\infty}\,\Bigl(\|\mu_j\|^2_{g_\alpha}-2\bigl\langle\vartheta_{g_\alpha}^F,\mu_j\bigr\rangle_{g_\alpha}\Bigr)=
\|\vartheta^F_{g_\alpha}\|^2_{g_\alpha}-2\|\vartheta^F_{g_\alpha}\|^2_{g_\alpha}=-\|\vartheta^F_{g_\alpha}\|^2_{g_\alpha},\end{align*}
the last but one equality being obtained from (\ref{muj}) and $\|\tau_j\|_{g_\alpha}<1/j$ by making use of the Cauchy--Schwarz (Bunyakovski) inequality. This implies that the sequence $(\mu_j)$ is minimizing in problem (\ref{sol}), i.e.\ $(\mu_j)\in\mathbb M_{g_\alpha,f}(F)$, and hence it converges strongly and vaguely in $\mathcal E^+_{g_\alpha}(F)$ to the (unique) extremal measure
$\xi:=\xi_{F,f}$ (Lemma~\ref{l-extr}). Since any compact subset of $D$ is contained in $U_j$ for all $j$ large enough, $\tau_j\to0$ vaguely in consequence of the former relation in (\ref{prop}), whence $\xi=\vartheta^F_{g_\alpha}$. Thus $\xi(D)=q<1$, and applying Corollary~\ref{l-extr5} shows that problem (\ref{sol}) is indeed unsolvable.

It thus remains to analyze the case where $\vartheta_{g_\alpha}^F(D)>1$. Assume first that $C_\xi\geqslant0$, $C_\xi$ being introduced by means of (\ref{Cxi}). Then, according to (\ref{e-pot1}),
\[U^\xi_{g_\alpha}\geqslant C_\xi+U^\vartheta_{g_\alpha}\quad\text{q.e.\ on $F$},\]                                                                                                                           whence
\[U^\xi_{g_\alpha}\geqslant U^{\vartheta^F_{g_\alpha}}_{g_\alpha}\quad\text{q.e.\ on $F$}.\]
Applying to $\xi,\vartheta^F_{g_\alpha}\in\mathcal E^+_{g_\alpha}(F)$ the refined principle of positivity of mass for $g_\alpha$-potentials as stated in \cite[Theorem~5.1]{Z-AMP}, which is possible due to (\ref{necsuf}), we therefore get
\[\xi(D)\geqslant\vartheta_{g_\alpha}^F(D)>1,\]
which however contradicts (\ref{ext-eq4}). Thus necessarily
\begin{equation}\label{Cxi2}
C_\xi<0.
\end{equation}

Integrating (\ref{e-pot1}) with respect to $\xi$, and then substituting (\ref{Cxi}) into the inequality thereby obtained, we have
\[C_\xi=\int U^\xi_{g_\alpha,f}\,d\xi\geqslant C_\xi\cdot\xi(D),\]
whence $\xi(D)\geqslant1$ in view of (\ref{Cxi2}). Combined with (\ref{ext-eq4}) this implies that, actually, $\xi(D)=1$, and an application of Corollary~\ref{l-extr5} then shows that problem (\ref{sol}) is indeed solvable with $\xi$ serving as the solution $\lambda_{F,f}$. The proof is complete.

\section{Proof of Theorem~\ref{sup-desc}}\label{pr-sup-desc} Unless $\alpha<2$, assume $\Omega$ is connected. Also assume that either
(\ref{C1}) or (\ref{C2}) is fulfilled. Then according to Theorems~\ref{sol-1} and \ref{sol-fin}, the solution $\lambda_{F,f}$ to problem (\ref{sol}) does exist, and moreover
\begin{equation}\label{RRR''}\lambda_{F,f}=\left\{
\begin{array}{cl}\vartheta^F_{g_\alpha}&\text{if \ (\ref{C1}) holds},\\
\vartheta^F_{g_\alpha}+c_{F,f}\gamma_{F,g_\alpha}&\text{otherwise},\\ \end{array} \right.
\end{equation}
where $c_{F,f}$, the $f$-weighted equilibrium constant for the set $F$, is ${}\geqslant0$.

On account of (\ref{bal2}), we conclude from \cite[Theorem~7.2]{Z-bal}, providing a description of the support $S(\mu^Q_{\kappa_\alpha};\mathbb R^n)$, where $Q\subset\mathbb R^n$ is closed and $\mu\in\mathfrak M^+(\mathbb R^n)$, that
\begin{equation}\label{desc1}S(\vartheta^F_{g_\alpha};D)=\left\{
\begin{array}{cl}\check{F}&\text{if \ $\alpha<2$},\\
\partial_D\check{F}&\text{otherwise}.\\ \end{array} \right.
\end{equation}
In case $c_{g_\alpha}(F)<\infty$, the same formula (\ref{desc1}) holds true for the $g_\alpha$-equilibrium measure $\gamma_{F,g_\alpha}$ in place of $\vartheta^F_{g_\alpha}$, which can be seen with the aid of arguments similar to those in \cite[Proof of Theorem~7.2]{Z-bal}, now applied to the functions
$U_{\kappa_\alpha}^{\gamma_{F,g_\alpha}}$ and $1+U_{\kappa_\alpha}^{(\gamma_{F,g_\alpha})^Y_{\kappa_\alpha}}$,  $\alpha$-superharmonic and $\alpha$-harmonic on $D$, respectively. While doing that, we utilize a description of the $g_\alpha$-equilibrium potential $U^{\gamma_{F,g_\alpha}}_{g_\alpha}$, provided by \cite[Lemma~6.5]{Z-AMP}, as well as the representation (\ref{g1}) for $g_\alpha$-potentials of extendible measures on $D$.

In view of (\ref{RRR''}), all this results in (\ref{RRR}), thereby completing the proof.

\section{Proof of Theorem~\ref{continuity}}\label{pr-cont} As noticed in Section~\ref{sec-st}, problem (\ref{sol}) is (uniquely) solvable for every $K\in\mathfrak C_F$. Moreover, it is shown in the proof of Lemma~\ref{l-pot}, see (\ref{min-net}), that those solutions form a minimizing net, i.e.\ $(\lambda_{K,f})_{K\in\mathfrak C_F}\in\mathbb M_{g_\alpha,f}(F)$. Thus, by virtue of Lemma~\ref{l-extr},
\begin{equation}\label{m-extrKF}
\lambda_{K,f}\to\xi\quad\text{strongly and vaguely in $\mathcal E_{g_\alpha}^+(F)$ as $K\uparrow F$},
\end{equation}
$\xi:=\xi_{F,f}$ being the extremal measure. On the other hand, problem (\ref{sol}) is assumed to be solvable; therefore, according to Corollary~\ref{l-extr5},
\begin{equation}\label{xiL}
\xi=\lambda_{F,f},
\end{equation}
which substituted into (\ref{m-extrKF}) leads to (\ref{Cont}).

In the same manner as in the proof of (\ref{ltoxip}), we deduce from (\ref{m-extrKF}) that there exists a subsequence $(K_j)$ of the net $(K)_{K\in\mathfrak C_A}$ such that
\[U^{\lambda_{K_j,f}}_{g_\alpha}\to U^\xi_{g_\alpha}\quad\text{pointwise q.e.\ on $D$ as $j\to\infty$,}\]
which combined with (\ref{xiL}) results in (\ref{Cont2}).

It follows from (\ref{cc}) that
\[c_{K,f}=\int U_{g_\alpha,f}^{\lambda_{K,f}}\,d\lambda_{K,f}=\|\lambda_{K,f}\|^2_{g_\alpha}+\int U_{g_\alpha}^\vartheta\,d\lambda_{K,f}=
\|\lambda_{K,f}\|^2_{g_\alpha}+\bigl\langle\vartheta^F_{g_\alpha},\lambda_{K,f}\bigr\rangle_{g_\alpha},\]
and similarly
\[c_{F,f}=
\|\lambda_{F,f}\|^2_{g_\alpha}+\bigl\langle\vartheta^F_{g_\alpha},\lambda_{F,f}\bigr\rangle_{g_\alpha}.\]
By the strong convergence of $(\lambda_{K,f})_{K\in\mathfrak C_F}$ to $\lambda_{F,f}$, this yields (\ref{Cont3}).

To complete the proof, assume now that $\vartheta^F_{g_\alpha}(F)\leqslant1$. For any two relatively closed $F_1,F_2\subset D$ such that $F_1\subset F_2$, and any bounded $\mu\in\mathfrak M^+(D\setminus F_2;D)$, we have \[\mu_{g_\alpha}^{F_1}=\bigl(\mu_{g_\alpha}^{F_2}\bigr)_{g_\alpha}^{F_1}\]
(balayage "with a rest", see \cite[Corollary~2.5]{Z-AMP}), and therefore, by virtue of (\ref{bal4g}),
\[\mu_{g_\alpha}^{F_1}(D)\leqslant\mu_{g_\alpha}^{F_2}(D).\]
Thus the net $\bigl(\vartheta^K_{g_\alpha}(D)\bigr)_{K\in\mathfrak C_F}$ increases and does not exceed $\vartheta^F_{g_\alpha}(D)$  $\bigl({}\leqslant1\bigr)$, which implies that Theorem~\ref{sol-fin} is applicable to any $K\in\mathfrak C_F$. This gives
\[c_{K,f}=\frac{1-\vartheta^K_{g_\alpha}(D)}{c_{g_\alpha}(K)},\]
and so $(c_{K,f})_{K\in\mathfrak C_F}$ decreases. In view of (\ref{Cont3}), we get (\ref{Cont3'}), whence the theorem.

\section{Proof of Theorem~\ref{continuity2}}\label{pr-cont2}

We first observe from the monotonicity of the net $(F_s)_{s\in S}$ that $\bigl(w_{g_\alpha,f}(F_s)\bigr)_{s\in S}$ is bounded and increasing, and moreover
\begin{equation}\label{lp1}
 \lim_{s\in S}\,w_{g_\alpha,f}(F_s)\leqslant w_{g_\alpha,f}(F).
\end{equation}

Since both $c_{g_\alpha}(F_{s_1})$ and $c_{g_\alpha}(F)$ are nonzero and finite, so are $c_{g_\alpha}(F_s)$ for all $s\geqslant s_0$, whence the minimizers $\lambda_{F_s,f}$ do exist (Theorem~\ref{sol-fin}). Furthermore, for any $t\geqslant s_0$,
\[(\lambda_{F_s,f})_{s\geqslant t}\subset\mathcal E^+_{g_\alpha}(F_t,1).\]
By the convexity of $\mathcal E^+_{g_\alpha}(F_t,1)$, $(\lambda_{F_t,f}+\lambda_{F_s,f})/2\in\mathcal E^+_{g_\alpha}(F_t,1)$ for all  $s\geqslant t$, hence
\[4w_{g_\alpha,f}(F_t)\leqslant4I_{g_\alpha,f}\biggl(\frac{\lambda_{F_t,f}+\lambda_{F_s,f}}{2}\biggr)=
\|\lambda_{F_t,f}+\lambda_{F_s,f}\|_{g_\alpha}^2+4\int f\,d(\lambda_{F_t,f}+\lambda_{F_s,f}).\]
Using the parallelogram identity in the pre-Hilbert space $\mathcal E_{g_\alpha}$ therefore gives
\begin{align*}
0\leqslant\|\lambda_{F_t,f}-\lambda_{F_s,f}\|_{g_\alpha}^2&\leqslant-4w_{g_\alpha,f}(F_t)+2I_{g_\alpha,f}(\lambda_{F_t,f})+
2I_{g_\alpha,f}(\lambda_{F_s,f})\\{}&=2w_{g_\alpha,f}(F_s)-2w_{g_\alpha,f}(F_t),
\end{align*}
whence the net $(\lambda_{F_s,f})_{s\geqslant t}\subset\mathcal E^+_{g_\alpha}(F_t,1)$ is strongly Cauchy. Since $\mathcal E^+_{g_\alpha}(F_t,1)$ is complete in the (induced) strong topology (cf.\ Theorem~\ref{th-comp} with $\kappa:=g_\alpha$),  $(\lambda_{F_s,f})_{s\geqslant t}$ converges strongly (hence vaguely, the kernel $g_\alpha$ being perfect) to some unique $\lambda\in\mathcal E^+_{g_\alpha}(F_t,1)$. Noting that this holds for each $t\geqslant s_0$, we thus get $\lambda\in\mathcal E^+_{g_\alpha}(F,1)$, and therefore
\begin{equation}\label{Lim}
\lim_{s\in S}\,w_{g_\alpha,f}(F_s)=\lim_{s\in S}\,I_{g_\alpha,f}(\lambda_{F_s,f})=I_{g_\alpha,f}(\lambda)\geqslant w_{g_\alpha,f}(F),
\end{equation}
the latter equality being derived from the strong continuity of $I_{g_\alpha,f}(\cdot)$ on $\mathcal E_{g_\alpha}(F_{s_1})$ (Lemma~\ref{str-cont}). Combining (\ref{Lim}) with (\ref{lp1}) proves (\ref{wup}) as well as $\lambda=\lambda_{F,f}$.

It has thus been shown that $\lambda_{F_s,f}\to\lambda_{F,f}$ strongly and vaguely in $\mathcal E_{g_\alpha}^+$ as $s$ ranges through $S$. The rest of the proof runs in a way similar to that in the proof of Theorem~\ref{continuity} (see the last three paragraphs in the preceding section).

\section{Proof of Theorem~\ref{th-ex}}\label{th-ex-pr}

To begin with, we first observe that under the requirements of Theorem~\ref{th-ex}, the set $F$ must be unbounded in $\mathbb R^n$, whence the following lemma holds true.

\begin{lemma}\label{l-inf}We have
\begin{equation}\label{eq-l-inf}
\lim_{|x|\to\infty, \ x\in F}\,U^\vartheta_{\kappa_\alpha}(x)=0.
\end{equation}
\end{lemma}

\begin{proof} This only needs to be verified when $S(\vartheta;D)$ is unbounded in $\mathbb R^n$. Since $\vartheta(D)$ is finite, for any $\varepsilon\in(0,\infty)$, there is an open neighborhood $U$ of $\infty_{\mathbb R^n}$ such that
\[\vartheta(U)<\varepsilon\varrho^{n-\alpha},\]
$\varrho\in(0,\infty)$ being defined by means of (\ref{dist}), whence
\begin{equation}\label{eq-l-inf1}
\int\kappa_\alpha(x,y)\,d\vartheta|_U(y)<\varepsilon\quad\text{for all $x\in F$}.
\end{equation}
On the other hand, one can choose a neighborhood $U'$ of $\infty_{\mathbb R^n}$ so that $U'\subset U$ and
\[\int\kappa_\alpha(x,y)\,d(\vartheta-\vartheta|_U)(y)<\varepsilon\quad\text{for all $x\in U'$},
\]
which together with (\ref{eq-l-inf1}) proves (\ref{eq-l-inf}).
\end{proof}

We further note that, since $c_{\kappa_\alpha}(Y)=0$ while
$F$ is not $\alpha$-thin at infinity,
\begin{equation}\label{harm1}\omega_\alpha(x,\overline{\mathbb R^n}\setminus D;\Omega)=0\quad\text{for all $x\in\Omega$},\end{equation}
cf.\ definitions (\ref{def-h}), (\ref{Def}) and Corollary~\ref{cor-infty}. Therefore, according to Theorem~\ref{sol-infin}, the minimizer $\lambda_{F,f}$ exists if and only if $\vartheta(D)\geqslant1$. Furthermore, if $\vartheta(D)=1$, then, in consequence of (\ref{harm1}), $\vartheta^F_{g_\alpha}(D)=1$ (see Theorem~\ref{th-balM2}), and applying Theorem~\ref{sup-desc} shows that a description of $S(\lambda_{F,f};D)$ is given, indeed, by means of (\ref{RRR}).

It remains to analyze the case $\vartheta(D)>1$. Assume, to the contrary, that $S(\lambda_{F,f};D)$ is unbounded, and so there is a sequence $(x_j)\subset S(\lambda_{F,f};D)\subset F$ approaching $\infty_{\mathbb R^n}$. Noting that $U_{g_\alpha}^\vartheta\bigl|_F$ is continuous while $U_{g_\alpha}^{\lambda_{F,f}}\bigl|_F$ is l.s.c., we deduce from (\ref{2}) that
\begin{equation}\label{desc7}
U_{g_\alpha,f}^{\lambda_{F,f}}(x_j)=U_{g_\alpha}^{\lambda_{F,f}}(x_j)-U_{g_\alpha}^\vartheta(x_j)\leqslant c_{F,f}\quad\text{for all $j$},
\end{equation}
$c_{F,f}$ being the $f$-weighted equilibrium constant. According to Corollary~\ref{l-extr5},
\[\lambda_{F,f}=\xi_{F,f}=:\xi,\]
where $\xi_{F,f}$ is the extremal measure, introduced in Lemma~\ref{l-extr}. Comparing (\ref{cc}) and (\ref{Cxi}) therefore gives
$c_{F,f}=C_\xi$, whence, by virtue of (\ref{Cxi2}),
\[c_{F,f}<0.\]
Substituting this into (\ref{desc7}) and then letting $j\to\infty$, we see from Lemma~\ref{l-inf} that
\[\liminf_{x\to\infty_{\mathbb R^n},\ x\in F}\,U_{\kappa_\alpha}^{\lambda_{F,f}}(x)=\liminf_{x\to\infty_{\mathbb R^n},\ x\in F}\,U_{g_\alpha}^{\lambda_{F,f}}(x)<0,\]
the equality being valid since $g^\alpha_D(x,y)=\kappa_\alpha(x,y)$ for all $(x,y)\in D\times\mathbb R^n$ because of the assumption $c_{\kappa_\alpha}(Y)=0$. This, however, contradicts \cite[Remark~4.12]{KM}, the set $F$ not being $\alpha$-thin at infinity. The proof is complete.

\section{Acknowledgements} The author thanks Krzysztof Bogdan, Vladimir Eiderman, and Artur Rutkowski for useful discussions on the content of the paper.

\section{A data availability statement} This manuscript has no associated data.

\section{Funding and Competing interests} The author has no relevant financial or non-financial interests to disclose.

\end{document}